\documentclass[12pt, a4paper]{amsart}
\usepackage[english]{babel}
\usepackage{amsmath,enumerate, amsfonts, amssymb,amsthm}
\usepackage[dvips]{graphics} 
\usepackage{xcolor}
\usepackage{graphicx}
\usepackage{amscd}
\usepackage{hyperref}
\usepackage{cleveref}

\newtheorem{thm}{Theorem}[section]
\newtheorem{lemma}[thm]{Lemma}
\newtheorem{prop}[thm]{Proposition}
\newtheorem{cor}[thm]{Corollary}

\theoremstyle{definition}
\newtheorem{defi}[thm]{Definition}

\newtheorem{rmk}[thm]{Remark}

\usepackage[utf8x]{inputenc}

\DeclareMathOperator{\dist}{\rm {dist }}

\usepackage{enumitem}
\setlist[itemize]{labelindent=.6em, itemindent=1em, leftmargin=!, label=\textbullet}

\usepackage{ifmtarg}
\makeatletter
  \newcommand{\TODO}[1]{\@ifmtarg{#1}{\emph{\textcolor{red}{\textbf{TODO}}}~}{\textcolor{red}{ \emph{\textbf{TODO:}~#1~}}}}
\makeatother

\newcommand{\R}{\mathbb R}

\newcommand{\N}{\mathbb N_{\ge0}}

\newcommand{\ord}{\operatorname{ord}}

\newcommand{\C}{\mathbb{C}}

\makeatletter
\newcommand{\vast}{\bBigg@{4}}
\newcommand{\Vast}{\bBigg@{5}}
\makeatother

\usepackage{ifthen}
\newcommand{\clos}[2][]{\ensuremath{\overline{#2}\ifthenelse{\equal{#1}{}}{}{^{#1}}}}

\newcommand{\W}{\mathcal W}
\newcommand{\PW}{\mathcal {PW}}
\newcommand{\QW}{\mathcal {QW}}

\newcommand{\horn}{\mathcal {W}}

\newcommand{\XX}{\mathcal {X}}
\newcommand{\SX}{\mathcal {F}}

\newcommand{\Ph}{\Phi}
\newcommand{\Phh}{\tilde \Phi}

\title{Lipschitz stratification of complex hypersurfaces in codimension 2}

\author{Adam Parusi\'nski and Lauren\c tiu P\u aunescu}
\address {Universit\'e C\^ote d'Azur,  CNRS,  LJAD, UMR 7351, 06108 Nice, France}
\email{adam.parusinski@univ-cotedazur.fr>}

\address{School of Mathematics and Statistics, The University of Sydney,
  Sydney, NSW, 2006, Australia }%
\email{laurentiu.paunescu@sydney.edu.au}%

\thanks{The first author is grateful for the support and hospitality of the Sydney Mathematical Research Institute (SMRI). 
Partially supported  ANR project LISA (ANR-17-CE40-0023-03).
}
\RequirePackage[normalem]{ulem} 
\RequirePackage{color}\definecolor{RED}{rgb}{1,0,0}\definecolor{BLUE}{rgb}{0,0,1} 
\providecommand{\DIFaddbegin}{} 
\providecommand{\DIFaddend}{} 
\providecommand{\DIFdelbegin}{} 
\providecommand{\DIFdelend}{} 
\providecommand{\DIFaddbeginFL}{} 
\providecommand{\DIFaddendFL}{} 
\providecommand{\DIFdelbeginFL}{} 
\providecommand{\DIFdelendFL}{} 
\newcommand{\DIFscaledelfig}{0.5}
\RequirePackage{settobox} 
\RequirePackage{letltxmacro} 
\newsavebox{\DIFdelgraphicsbox} 
\newlength{\DIFdelgraphicswidth} 
\newlength{\DIFdelgraphicsheight} 
\LetLtxMacro{\DIFOincludegraphics}{\includegraphics} 
\newcommand{\DIFaddincludegraphics}[2][]{{\color{blue}\fbox{\DIFOincludegraphics[#1]{#2}}}} 
\newcommand{\DIFdelincludegraphics}[2][]{
\sbox{\DIFdelgraphicsbox}{\DIFOincludegraphics[#1]{#2}}
\settoboxwidth{\DIFdelgraphicswidth}{\DIFdelgraphicsbox} 
\settoboxtotalheight{\DIFdelgraphicsheight}{\DIFdelgraphicsbox} 
\scalebox{\DIFscaledelfig}{
\parbox[b]{\DIFdelgraphicswidth}{\usebox{\DIFdelgraphicsbox}\\[-\baselineskip] \rule{\DIFdelgraphicswidth}{0em}}\llap{\resizebox{\DIFdelgraphicswidth}{\DIFdelgraphicsheight}{
\setlength{\unitlength}{\DIFdelgraphicswidth}
\begin{picture}(1,1)
\thicklines\linethickness{2pt} 
{\color[rgb]{1,0,0}\put(0,0){\framebox(1,1){}}}
{\color[rgb]{1,0,0}\put(0,0){\line( 1,1){1}}}
{\color[rgb]{1,0,0}\put(0,1){\line(1,-1){1}}}
\end{picture}
}\hspace*{3pt}}} 
} 
\LetLtxMacro{\DIFOaddbegin}{\DIFaddbegin} 
\LetLtxMacro{\DIFOaddend}{\DIFaddend} 
\LetLtxMacro{\DIFOdelbegin}{\DIFdelbegin} 
\LetLtxMacro{\DIFOdelend}{\DIFdelend} 
\DeclareRobustCommand{\DIFaddbegin}{\DIFOaddbegin \let\includegraphics\DIFaddincludegraphics} 
\DeclareRobustCommand{\DIFaddend}{\DIFOaddend \let\includegraphics\DIFOincludegraphics} 
\DeclareRobustCommand{\DIFdelbegin}{\DIFOdelbegin \let\includegraphics\DIFdelincludegraphics} 
\DeclareRobustCommand{\DIFdelend}{\DIFOaddend \let\includegraphics\DIFOincludegraphics} 
\LetLtxMacro{\DIFOaddbeginFL}{\DIFaddbeginFL} 
\LetLtxMacro{\DIFOaddendFL}{\DIFaddendFL} 
\LetLtxMacro{\DIFOdelbeginFL}{\DIFdelbeginFL} 
\LetLtxMacro{\DIFOdelendFL}{\DIFdelendFL} 
\DeclareRobustCommand{\DIFaddbeginFL}{\DIFOaddbeginFL \let\includegraphics\DIFaddincludegraphics} 
\DeclareRobustCommand{\DIFaddendFL}{\DIFOaddendFL \let\includegraphics\DIFOincludegraphics} 
\DeclareRobustCommand{\DIFdelbeginFL}{\DIFOdelbeginFL \let\includegraphics\DIFdelincludegraphics} 
\DeclareRobustCommand{\DIFdelendFL}{\DIFOaddendFL \let\includegraphics\DIFOincludegraphics} 

\begin{document}
\begin{abstract}
We show that the Zariski canonical stratification of complex hypersurfaces 
is locally bi-Lipschitz trivial along the strata of codimension two.  
More precisely, we study the Zariski equisingular families of surface, not necessarily isolated,  singularities in $\C^3$.   We show that a natural stratification of such a family, given by the singular set and the generic family of polar curves, provides a Lipschitz stratification in the sense of Mostowski. 
 In particular such families are bi-Lipschitz trivial by  trivializations 
obtained by  integrating   Lipschitz vector fields.  
\end{abstract}
\maketitle
\tableofcontents

\section{Introduction}

In the geometric study of complex singular algebraic varieties or analytic spaces the notion of stratification plays an essential role.  It is well known that there always exists a stratification  that is topologically equisingular (i.e. trivial) along each stratum.  This is usually achieved by means of a Whitney stratification.  Another and entirely independent way of constructing such a stratification is Zariski equisingularity.  A desirable  important feature is the existence of a stratification that satisfies stronger equisingularity properties than the one given by Whitney Conditions.  This is known about Zariski (generic) equisingularity, though its precise geometric properties are still to be understood.  For instance, it is well known that Zariski equisingular families of plane curve singularities are bi-Lipschitz trivial.  The goal of this paper is to extend this observation to the next case, the families of surface singularities in $\C^ 3$.

In 1979 O. Zariski \cite{Zariski1979} presented a general theory of equisingularity for algebroid and  algebraic hypersurfaces over an algebraically closed field of characteristic zero.  Zariski's theory is based on the notion of equisingularity  along the strata defined by considering the discriminants loci of successive 
"generic" projections. This concept, now referred to as Zariski equisingularity 
or generic Zariski equisingularity,  
was called by Zariski himself algebro-geometric equisingularity, since it is defined by purely algebraic means but reflects several natural geometric properties. In \cite{zariski65-S2} Zariski studied the case of strata of codimension one.  In this case the hypersurface is locally isomorphic to an equisingular (topologically trivial if the ground field is $\C$) family  of plane curve singularities. Moreover, by Theorem 8.1 of \cite{zariski65-S2}, Zariski's stratification satisfies Whitney's conditions along the strata of codimension one, and over $\C$, by  \cite{PTpreprint},  such an equisingular family of plane curves is bi-Lipschitz trivial, i.e. trivial by a local ambient bi-Lipschitz homeomorphism. 
In general, Zariski equisingularity  implies Whitney's conditions as shown by Speder \cite{speder75}. For a survey on Zariski equsingularity and its recent applications see \cite{P2020}.

 In 1985 T. Mostowski \cite{mostowski85} introduced the notion of Lipschitz stratification of complex analytic spaces or algebraic varieties, by imposing  local conditions on tangent spaces to the strata, stronger than Whitney's conditions. 
 Mostowski's work was partly motivated by the question of Siebenmann and Sullivan \cite{SS79} whether the number of local Lipschitz types on (real or complex) analytic spaces is countable.  Mostowski's Lipschitz stratification satisfies the extension property of  stratified vector fields from lower dimensional to higher dimensional strata, and therefore implies local bi-Lipschitz triviality.  Its  construction is similar to the one  of Zariski, but involves considering many projections at each stage of construction.  It is related to the geometry of polar varieties, as shown by  Mostowski in the case of hypersurface singularities in $\C^3$, see \cite{mostowski88}.  In general, the construction of a Lipschitz stratification is complicated and involves many stages. It was conjectured by J.-P. Henry and T. Mostowski that Zariski equisingular families of surface singularities in $\C^3$ admit natural Lipschitz stratification by taking the singular locus and the family of "generic" polar curves as strata.  We show this conjecture in this paper, see Theorem \ref{thm:maintheorem}.

Recent works, see for instance \cite{V2005}, \cite{NV2016}, \cite{BFS2018}, \cite{halupczok-yin2108} \cite{kovacsics-halupczok2021}, show further development and  progress on understanding the Lipschitz structure of singularities and its relation to other geometric phenomena appearing in 
the study of local properties of complex or real analytic or algebraic singular spaces.
  Among the major results and contributions we mention only the most important ones related to this paper, \cite{BNP2014} where the case of the "inner" metric was considered and \cite{NPpreprint} where the equivalence of Zariski Equisingularity and Lipschitz triviality for families of complex normal surface singularities was announced. 

 Our proof of Theorem \ref{thm:maintheorem} is based on local parameterizations of two geometric objects associated to such families: the polar wedges and the quasi-wings.  
 Both originate from the classical wings introduced by Whitney in \cite{whitneyannals65}.  The polar wedges are neighborhoods of families of polar curves, the critical loci of  corank-one projections. The quasi-wings,  originally introduced in \cite{mostowski85}, 
 are neighbourhoods of curves on which this projection is regular (with a control on the derivatives). Their local parameterizations, interesting by themselves, in the case of polar wedges originate from \cite{brianconhenry80} and \cite{teissierRabida82} and were recently considered  in \cite{NPpreprint}.  
 
  As we show the quasi-wings and the polar wedges cover a neighbourhood of the singularity.  The proof of this fact follows from the
 analytic wings construction of \cite{PP17}.  

The definition of  "generic projection" is crucial for Zariski's theory.
 Zariski's study of codimension one singularities (families of plane curve singularities) required just transverse projections. 
  This is no longer the case for singularities in codimension 2. 
  In \cite{luengo85} Luengo gave an example of a family of surface 
singularities in $\C^3$ that is Zariski equisingular for one transverse 
projection but not for a generic transverse projection. 
Therefore we make precise what we mean by "generic projection" in our context and  we state it in our Transversality Assumptions.  This is important since this  condition can be computed and algorithmically verified. 

\medskip
\noindent
\textbf{Acknowledgment.}\\
The authors would like to thank the referee for many valuable remarks and suggestions that significantly  improved our paper.

\section{Set-up and statement of results}
Let  $f(x,y,z,t): (\C^{3+l},0) \to (\C, 0)$ be analytic.  We suppose that $f(0,0,0,t) =0$ for every $t \in (\C^l,0),$ and regard  $f$ as an analytic family $f_t(x,y,z)= f(x,y,z,t)$ of analytic function germs parameterized by $t$. 
In what follows we suppress for simplicity the germ notation. 

We denote  by $\mathcal X= f^{-1}(0)$ and by $\Sigma_f$ the singular set of $\mathcal X$.  We always assume that the germs $f_t$ are reduced, 
 and that the system of coordinates is sufficiently generic (see the Transversality Assumptions below for a precise formulation).  In particular we assume that the restriction of the projection 
$\pi (x,y, z, t) = (x,y, t)$ to $\mathcal X$ is finite. 

 Denote 
by $C_f$ the polar set of $\pi|_{\mathcal X}$,  i.e. the closure of the critical locus of the projection $\pi $ restricted to the regular part of $\mathcal X$.  
The set $C_f$ can be understood as a family of space curves (polar curves) parameterized by $t$. Let   
\begin{align}
S = \{  f(x,y,z; t)=  f'_z (x,y,z; t) = 0\} = \Sigma_f \cup C_f .
\end{align}
The main goal of this paper is to show the following result, Theorem  \ref{thm:maintheorem} (for the notion of Zariski equisingular families of hypersurface singularities in $(\C^3,0)$ see the next subsection 
\ref{ss:Zariski equisingularity},  for Mostowski's  Lipschitz stratification see 
subsection \ref{ss:lipschitz stratification}).

\begin{thm}\label{thm:maintheorem}
Suppose that the family $\mathcal X_t = f_t^{-1} (0)$ is generically linearly Zariski equisingular.  Then it is bi-Lipschitz trivial.  That is,  there are neighbourhoods $\Omega$ of $0$ in $\C^3\times \C^l$, $\Omega_0$ of $0$ in $\C^3$, and $U$ of $0$ in $\C^l$, and a bi-Lipschitz homeomorphism 
$$\Phi : \Omega_0 \times U \to  \Omega, $$
satisfying $\Phi (x,y,z,t) = (\Psi (x,y,z,t) , t)$, $\Phi (x,y,z,0) = (x,y,z,0)$, 
such that 
$$\Phi (\mathcal X_0\times U ) ={} \mathcal X.$$

Moreover, $\{\mathcal X \setminus S, S\setminus T, T\}$, where $T=\{0\}\times \C^l$,  
defines a Lipschitz stratification of $\mathcal X$ in the sense of Mostowski.  In particular
the homeomorphism $\Phi$ can be obtained by  integration of Lipschitz vector fields.  
\end{thm} {}

The non-parameterized version, i.e. if $l=0$,  of Theorem \ref{thm:maintheorem} was proven in \cite{mostowski88}, and the general version, as stated above, was conjectured by J.-.P Henry and T. Mostowski more than twenty years ago.  The bi-Lipschitz triviality for families of normal surface singularities in $\C^3$ was announced in \cite{NPpreprint}.  
Our proof uses some  ideas of \cite{NPpreprint} and \cite{BNP2014}, in particular that of polar wedges.  Nevertheless, our main idea of proof is different from that of \cite{NPpreprint}.  Moreover, we show a much stronger bi-Lipschitz property, the existence of a Lipschitz stratification in the sense of Mostowski.  This implies that the trivialization $\Phi$ can be obtained by integration of Lipschitz vector fields.  There is a difference between arbitrary bi-Lipschitz trivializations, and the ones obtained by integration of Lipschitz vector fields (note that the bi-lipschitz trivializations of \cite{BNP2014}, \cite{NPpreprint}, 
\cite{V2005} do not satisfy this property). For instance the latter one implies the continuity of the Gaussian curvature, see  \cite{mostowski85} section 10 and \cite{P2005}. 

The notion of Lipschitz stratification was defined by Mostowski in terms of regularity conditions on 
tangent spaces to strata, but to show that $\{\mathcal X \setminus S, S\setminus T, T\}$ is a Lipschitz stratification we do not use Mostowski's definition but an equivalent characterization based on the extension of stratified Lipschitz vector fields, see subsection \ref{ss:lipschitz stratification} below.  For this we use two, in a way, complementary constructions, the polar wedges of \cite{BNP2014} and \cite{NPpreprint} (covering neighbourhoods of the critical loci of a generic linear projection) and the quasi-wings of \cite{mostowski85} 
(covering their complements).

 Both can be understood 
as a generalized version of the classical wings. 
Actually we need a strong analytic form of the latter given by \cite{PP17}, in order to construct for an arbitrary real analytic arc, not contained in polar wedges, first a complex analytic wing and then a quasi-wing containing it, see Proposition \ref{prop:existenceq-wing}.  

Many parts of the proof are fairly technical.  In order to simplify the exposition we used the following strategy.  Virtually, for all the geometric constructions of the proof, including the description of the stratified Lipschitz vector fields on polar wedges in Proposition \ref{prop:LVFcriterion2} or on quasi-wings in Proposition \ref{prop:LVFcriterion2-qw}, the emphasis is given to the non-parameterized case, i.e., with $l=0$.  The profound understanding of this case, rightly stated, makes the parameterized case much easier.


\subsection{Zariski equisingularity}\label{ss:Zariski equisingularity}

Given a family of reduced analytic functions germs  $f_t(x,y,z): (\C^3,0)\to (\C, 0)$ 
as above, we denote by  $\Delta (x,y,t)$ the discriminant of the projection $\pi$ restricted to $\mathcal X$.  The zero set of $\Delta (x,y,t)$  is a family of plane curve singularities  parameterized by $t$.  We say that the family $\mathcal X_t$ is \emph{Zariski equisingular (with respect to the projection $\pi$)} if  $t\to \{\Delta (x,y,t)=0\}$ is an equisingular family of plane curves, that is satisfying one of the standard equivalent definitions, see \cite{zariski65-S1}, \cite[p. 623]{teissier77}.  We shall often use  the classical result 
saying that a family of equisingular plane curves admits a uniform Puiseux expansion with respect to parameters,  in the sense of \cite[Theorem 2.2]{PP17}.  
We refer to it as to the Puiseux with parameter theorem.

We say that the family $\mathcal X_t$ is \emph{generically linearly Zariski equisingular} if it is Zariski equisingular after a generic linear change of coordinates $x,y,z$. 

In the proof of Theorem \ref{thm:maintheorem} we 
use the following precise assumptions on $f$, called Transversality Assumptions,  that are implied by the generic linear Zariski equisingularity.  

Let us denote by $\pi_b$ the projection $\C^3\times \C^l\to \C^ 2\times \C^l$ parallel to 
$(0,b,1,0)$, that is $\pi_b(x,y,z,t) = (x, y - bz,t)$.  
 We denote by  $\Delta_b (x,y,t)$ the discriminant 
 of the projection $\pi_b$ restricted to  $\mathcal X$.

\medskip
\noindent
\textbf{Transversality Assumptions.} 
The tangent cone $C_0(\mathcal X_0)$ to $\mathcal X_0= f_0^{-1} (0)$ does not contain the $z$-axis and, for $b$ and $t$ small, the family of the discriminant loci $\Delta_{b}=0$ is an equisingular family of plane curve singularities with respect to $b$ and $t$ as parameters.  
Moreover, we suppose that $\Delta_{0}=0$ is transverse to the $y$-axis 
and that $x=0$ is not a limit of tangent spaces to $\mathcal X_{reg}$, the regular part of $\mathcal X$.

\begin{rmk} 
Since Zariski equisingular families are  equimultiple, see \cite{zariski75} or 
\cite{PP17} [Proposition 1.13], the above assumptions imply  the following.  The tangent  
cone $C_0(\mathcal X_t)$  does not contain $(0,b,1)$, for $t$ and $b$ small. 
The $y$-axis is transverse to every $\{(x,y); \Delta_{b} (x,y,t)=0\}$, also for 
$t$ and $b$ small.  
\end{rmk}

We now show that a generically linearly Zariski equisingular family satisfies, after a linear change of coordinates $x,y,z$,  the Transversality Assumptions. First we need the following lemma.  

\begin{lemma}
The family $f_t(x,y,z)=0$ is generically linearly Zariski equisingular if and only if, after a linear change of coordinates $x,y,z$,  the family $f(x+az,y+bz,z,t)= 0$,  for $a,b,t$ small, 
is Zariski equisingular with respect to parameters $a,b,t$.    
\end{lemma}

\begin{proof}
The "if" part is obvious.  We show the "only if".  
 Let $\Delta (x,y,a,b,t)$ be the discriminant of $f(x+az,y+bz,z,t)$. By assumption there is an open subset $U\subset \C^2$ such that this family of plane curve germs  $\Delta (x,y,a,b,t)=0$ is equisingular with respect to $t$ for  every $(a,b)\in U$.
 Fix a small neighbourhood $V$ of the origin in $\C^l$ so that the subset 
 of parameters $(a,b,t)\in U\times V,$ such that $\Delta (x,y,a,b,t)=0$ changes the equisingularity type,  is a proper analytic subset of $Y\subset U\times V$.
 \textcolor{black}{The existence of such $Y$ follows for instance from  Zariski 
 \cite{zariski65-S1}, where it is shown that a family of plane curve singularities is equisingular if and only if its discriminant by a transverse 
 projection is equimultiple.(Equivalently, one may use semicontinuous invariants characterizing equisingularity such as the Milnor number for instance.)}     
 Then $Y$ cannot contain $U\times \{0\}$ (this would contradict the Zariski equisingularity of 
 $\Delta =0$ for $(a,b)\in U$ arbitrary and fixed).  Therefore, the family 
 $f(x+az,y+bz,z,t)= 0$ is Zariski equisingular for the parameters $a,b,t$ in a neighborhood of 
any point of 
 $(U\setminus Y)\times \{0\}$. This shows the claim.   
\end{proof}

Suppose now that the family $f_t=0$ is generically linearly Zariski equisingular and choose a generic line $\ell $ in the parameter space of $(a,b)\in U$ in the notation  of the proof of the above lemma.  The pencil of kernels of $\pi_{a,b}(x,y,z,t) = (x-at, y - bz,t)$, $(a,b)\in \ell$,  corresponds to  a hyperplane $H\subset \C^3$.  Choose coordinates $x,y,z$ so that $H= \{x=0\}$ and then $\ell$ corresponds to the pencil of projections parallel to $(0, b,1) \in H$.  Then in this system of  coordinates $(x,y,z),$ $f$ satisfies the Transversality Assumptions.


\subsection{Lipschitz stratification}\label{ss:lipschitz stratification}
In  \cite{mostowski85} T. Mostowski introduced a sequence of conditions on 
the tangent spaces to the strata of a stratified subset of $\C^n$ that, if satisfied, imply  the Lipschitz 
triviality of the stratification along each stratum.  Mostowski showed 
the existence of such stratifications for germs of complex analytic subsets of 
$\C^n$.  Note that there is no canonical Lipschitz stratification in the sense of 
Mostowski in general.  

For more information about the Lipschitz stratification we refer the interested reader to
 \cite{mostowski85}, \cite{Lipschitz-Fourier}, \cite{Lipschitz-review}, \cite{halupczok-yin2108}.

In \cite{mostowski88} Mostowski gave a criterion for a set to be a  codimension one stratum of 
Lipschitz stratification of a complex surface germ in $\C^3$, see the second example on pages 320-321 of \cite{mostowski88}.  This criterion implies that a generic polar curve 
can be chosen as such a stratum.  It is not difficult to complete Mostowski's
 argument and show Theorem \ref{thm:maintheorem} in the non-parameterized case ($l=0$).  
 In subsection \ref{ss:lip-strat-nonpar} we give a different proof which implies the parameterized case as well.  

Mostowski's conditions imply the existence 
of extensions of Lipschitz stratified vector fields from lower dimensional to 
higher dimensional strata, the property which, as shown in 
\cite{Lipschitz-Fourier}, is equivalent to Mostowski's conditions.  
Let us recall this equivalent definition.  
For this it is convenient to express Mostowski's stratification in terms of its skeleton, 
that is the union of strata of dimension $\le k$.  Let $X\subset \C^n$ be a complex analytic subset of dimension $d$ and let 
\begin{align}\label{eq:Lstratification}  X=X^d \supset X^{d-1} \supset \cdots \supset X^l \ne\emptyset ,
\end{align}
$l\ge 0$, $X^{l-1} =\emptyset$, be its filtration by complex analytic sets such that every $X^k \setminus X^{k-1}$ 
is either empty or nonsingular of pure dimension $k$.  

Our proof is based on the following characterization of Lipschitz stratification.

\begin{prop}[{\cite[Proposition 1.5]{Lipschitz-Fourier}}]\label{prop:LipStr-charcterization}
The filtration 
\eqref{eq:Lstratification} induces is a Lipschitz stratification if and only if one of the following 
equivalent conditions holds: 
\begin{enumerate}
\item [(i)]
There exists $C > 0$ such that for every $W \subset X$ satisfying $X^{j-1}\subseteq 
 W \subset X^j$, every Lipschitz stratified vector field on $W$ with a Lipschitz constant $L$, bounded on $W \cap X^l$ by $K$, can be extended to a Lipschitz stratified vector field on $X^j$ with a Lipschitz constant $C(L+K)$.
\item [(ii)] 
There exists $C > 0$ such that for every $W = X^{j-1}\cup \{q\}$, 
$q \in X^j$, each Lipschitz stratified vector field on $W$ with a Lipschitz constant $L$, bounded on $W \cap X^l$ by $K$, can be extended to a Lipschitz stratified vector field on $W\cup \{q'\}$, $q' \in X^j$, with a Lipschitz constant $C(L+K)$.
\end{enumerate} 
\end{prop}

Here by a stratified vector field we mean a vector field tangent to strata.  
In our particular case, stratification $\{\mathcal X \setminus S, S\setminus T, T\}$ it Lipschitz if and only if 
 there is a constant $C>0$ such that:
\begin{enumerate}
\item [(L1)]
for every couple of points $q,q'\in S\setminus T$, every stratified Lipschitz vector field on $T\cup \{q\}$, 
 with Lipschitz constant $L$ and bounded by $K$, can be extended to a Lipschitz 
 stratified vector field on $T\cup \{q,q'\}$ with Lipschitz constant $C (L+K) $.
 \item [(L2)]
for every couple of points $q,q'\in \mathcal X\setminus S$, every stratified Lipschitz vector field on $S\cup \{q\}$ with Lipschitz constant $L$ and bounded by $K$, can be extended to a Lipschitz vector field on $S\cup \{q,q'\}$ with Lipschitz constant $C (L+K)$.
\end{enumerate}

 In order to show the conditions (L1) and (L2) we consider two geometric constructions, 
the quasi-wings of Mostowski \cite{mostowski85} and the polar wedges of \cite{BNP2014} and \cite{NPpreprint}, that, as sets, together cover the whole 
$\mathcal X$.  We first show the (L1) condition in general and the (L2) condition on polar wedges.  This part of the proof is based on a complete description of the
stratified Lipschitz vector fields on polar wedges in terms of 
their parameterizations, see Section   \ref{sec:stratlip}.  
Note that in order to compare points on polar wedges we work with fractional powers, using parameterizations over the same allowable sector, see the Subsection \ref{ss:allowablesectors} for more details. 
In order to show (L2) on the quasi-wings we employ the following strategy.  
If Mostowski's conditions fail then they fail along real analytic arcs 
$\gamma(s), \gamma'(s)$, $s\in [0,\varepsilon)$, see \cite{mostowski85} Lemma 6.2 or the valuative  Mostowski's conditions of \cite{halupczok-yin2108}. 

For such arcs, however, if they are not in the union of polar wedges, we can construct quasi-wings containing them, 
say $\QW$ and $\QW'$ respectively, and then we show that  the  stratification  $\{\QW \cup \QW'  \setminus S, S\setminus T, T\}$ satisfies criterion (L2) on the arcs $\gamma(s), \gamma'(s)$.
For a precise statement and proof justifying this strategy the reader is referred to the rather technical Section \ref{sec:LVEC}.

\subsection{Notation and conventions}

In what follows we often use the following notations. For two complex function germs $f,g :(\C^k,0)\to (\C,0)$ we write :
\begin{enumerate}

\item
$|f (x)|\lesssim |g(x)|$ (or $ f =O(g)$) if  $|f(x)|\leq c|g(x)|,  c>0$ a given constant, in a neighbourhood of $ 0$. 

\item
$|f (x)|\sim |g(x)|$ if $|f(x)|\lesssim |g(x)|\lesssim |f(x)|$ 
in a neighbourhood of $ 0$.

\item
$|f (x)| \ll |g(x)|$ (or $f=o(g)$) if the ratio $\frac{|f(x)|}{|g(x)|}\to 0$ as $\|x\|\to 0$.

\end{enumerate}

While parameterizing analytic curve singularities  or families of such 
singularities in $\C^2$ and $\C^3$ using Puiseux Theorem, we ramify in variable $x=u^n$.  We often have to replace such an exponent $n$ by a multiple in order for such parameterizations to remain analytic, but we keep denoting it by $n$ for simplicity.  This makes no harm since we always work over an admissible sector as explained in subsection \ref{ss:allowablesectors}. By an analytic unit we mean a nowhere vanishing analytic function or a germ of such function.

\section{Families of polar curves} 

In this section we discuss how the families of polar curves of  
$\mathcal X$, associated to the projections $\pi _b$, $b\in \C$, 
depend to $b$.  The 
main result is Proposition \ref{prop:quadratic} (non-parameterized case) and Proposition \ref{prop:quadratic-t} (parameterized case).  The proposition in the non-parameterized case appeared  in the proof of the Polar wedge lemma, i.e. Proposition 3.4, of \cite{BNP2014}.  
The proofs of Propositions \ref{prop:quadratic} and \ref{prop:quadratic-t}
 are based on a key Lemma \ref{lem:keylemma1} due to   \cite{brianconhenry80} and  \cite{teissierRabida82}.

\subsection{Non-parameterized case} 
For simplicity  we first consider  the case of $f(x,y,z)$ without parameter.  
We assume that the coordinate system satisfies the Transversality Assumptions and therefore the family 
\begin{align}\label{eq:linearfamily-b}
F(X,Y,Z,b):= f(X,Y+bZ,Z), 
\end{align}
parameterized by  $b\in \C$ is Zariski equisingular for $b$ small.  
By this assumption the zero set of the discriminant $\Delta_F (X,Y,b)$ of $F$ satisfies the Puiseux with parameter theorem.  The set $F=F'_Z=0$, is the union 
$S_F=\Sigma_F\cup C_F$ of the singular set $\Sigma_F$ of $F$ and the family of the polar curves $C_F$.  
The set $S_F$  consists of finitely many irreducible components parameterized by 
\begin{align}\label{eq:parameterization1}
(u,b) \to (u^n, Y_i ( u,b), Z_i (u,b), b),
\end{align}
with $Y_i, Z_i$ analytic. Then $(u^n, Y = Y_i(u,b), b)$ parameterizes a component of the discriminant locus $\Delta_F = 0$ of $F$.  

The  below key lemma is a version of the first formula on page 278 of \cite{brianconhenry80}
or of a formula on page 465 of \cite{teissierRabida82}.

\begin{lemma}\label{lem:keylemma1}
\begin{align}\label{eq:mainidentity}
Z_i =  - \frac {\partial Y_i}  {\partial b} . 
\end{align}
\end{lemma}  

\begin{proof}
We have 
\begin{align}\label{eq:polar}
F (u^n,Y_i,Z_i,b) = 0=F'_Z (u^n,Y_i,Z_i,b) .
\end{align}
 We differentiate the first identity 
with respect to $b$ and use the second one to simplify the result 
\begin{align*}
0 =   F'_Y  \frac {\partial Y_i}  {\partial b}  + F'_Z  \frac {\partial Z_i}  {\partial b}  + 
 F'_b 
 = f'_y (u^n, Y_i+bZ_i, Z_i) \Bigl ( \frac {\partial Y_i}  {\partial b}  +Z_i \Bigr ).
\end{align*}
If $ f'_y (u^n, Y_i+bZ_i, Z_i)\not \equiv 0$ then  the formula \eqref{eq:mainidentity} holds.  Note that in this case  \eqref{eq:parameterization1} parameterizes an irreducible component of  $C_F$.   

If $f'_y (u^n, Y_i+bZ_i, Z_i)\equiv 0$ then, in addition to \eqref{eq:polar},  we have $F'_Y (u^n,Y_i,Z_i,b) = 0$.  Thus in this case  \eqref{eq:parameterization1} parameterizes a component of $\Sigma_F$.    
By the formula 
\begin{align}\label{eq:F'_z}
F'_Z (X,Y,Z,b ) = bf'_y (X,Y + bZ, Z) + f'_z (X,Y + bZ, Z),
\end{align}
$ (X,Y,Z,b)\in \Sigma_F$ if and only if $
(x,y,z) = (X,Y+bZ,Z) \in \Sigma_f$, the singular set of $f$. Thus in this case 
the map 
\begin{align}\label{eq:parameterization2}
(u,b) \to (u^n, y_i (u,b), z_i (u,b)), \qquad y_i  = Y_i +b Z_i, \, z_i= Z_i,
\end{align}
parameterizes a component of $\Sigma_f$.  Moreover,  
by the Transversality Assumptions, the projection of $\Sigma_f$ 
on the $x$-axis is finite. Consequently,  both 
$y_i = Y_i + bZ_i,$ and $Z_i$ are independent of $b$ and \eqref{eq:mainidentity} 
trivially holds.
\end{proof}

We note that, if $ f'_y (u^n, Y_i+bZ_i, Z_i)\not \equiv 0$, i.e. if  
 \eqref{eq:parameterization1} parameterizes a component 
of $C_F$, then \eqref{eq:parameterization2}  parameterizes a 
family of polar curves in $f^{-1}(0)$  defined by the projections $\pi_b$.  
In both cases,  the functions $y_i ( u,b)$, $z_i (u,b)=Z_i(u,b)$, and $Y_i ( u,b)$ are related by  
\begin{align}\label{eq:relations}
z_i = - \partial Y_i/\partial b,  \quad y_i  = Y_i +b z_i, \quad \partial y_i/\partial b = b\partial z_i/\partial b.
\end{align}
In particular, the expansion of $y_i$ cannot have a term linear in $b$. 

By the Zariski equisingularity assumption for any two distinct  branches  $
Y_i(u,b)$, $Y_j(u,b)$ there is $k_{ij}\in \N$ such that 
$Y_i(u,b)- Y_j(u,b) = u^{k_{ij}} unit (u,b)$. Note that, by the transversality with the $y-$axis, we have  $k_{ij}\geq n$ 
. By \eqref{eq:relations} this implies the following result.  

\begin{lemma}\label{lem:contacts1}
For $i\ne j$ There is $k_{ij}\in \N,  \, k_{ij}\geq n,$ such that 
\begin{align}\label{eq:generalformpairs}
& y_i(u,b)- y_j(u,b) = u^{k_{ij}} unit (u,b), \\
\notag
& z_i(u,b)- z_j(u,b) = O(u^{k_{ij}}) .
\end{align}
\end{lemma}

 The next result, that we will prove later in the more general parameterized case,  is crucial.  

\begin{prop}\label{prop:quadratic}
There are integers $m_i\in \N, \, m_i\geq n,$ such that  
\begin{align}\label{eq:generalform}
& y_i(u,b) = y_i(u,0) + b^2 u^{m_i} \varphi_i  (u,b), \\
\notag
& z_i(u,b) = z_i(u,0) + b u^{m_i} \psi_i (u,b),
\end{align}
with either $\varphi_i (0,0)\ne 0$, $\psi_i (0,0)\ne 0$ or, if \eqref{eq:parameterization2} parameterizes a component of $\Sigma_f$ then $\varphi_i\equiv \psi_i \equiv 0$.
\end{prop}


\subsection{Parameterized case} 
We extend the results of the previous subsection to the parameterized case 
 family 
\begin{align}\label{eq:linearfamily-bt}
F(X,Y,Z,b,t):= f(X,Y+bZ,Z,t), 
\end{align}
with $f$ satisfying the Transversality Assumptions.  
Thus $F$ is now Zariski equisingular with respect to the parameters $b$ and $t$ and therefore 
the discriminant 
$\Delta_f (X,Y,b, t)$ of $F$ with respect to $Z$ satisfies the Puiseux with parameter theorem.   
  Similarly to the non-parameterized case,  $S_F=\{F=F'_Z=0\}$ is parameterized by 
 \begin{align}\label{eq:parameterization1-t}
(u,b, t) \to (u^n, Y_i (u,b, t), Z_i (u,b, t), b, t),   
\end{align}
and consists of the singular locus $\Sigma _F$ and a family $C_F$ 
of polar curves, now parameterized  by $b$ and $t$. 

The lemma \ref{lem:keylemma1} still holds (with the same proof) so we have $Z_i =  - \partial Y_i/\partial b$.  Then 
\begin{align}\label{eq:parameterization2-t}
(u,b,t) \to p_i(u,b,t)= (u^n, y_i (u,b,t), z_i (u,b,t),t), \qquad y_i  = Y_i +b Z_i, \, z_i= Z_i.
\end{align}   
parameterize in $\C^3\times \C^l$ the families of polar curves with respect to 
the projections $\pi_b$ with $t$ being a parameter, or the branches of the singular locus $\Sigma_f$.   
The relations \eqref{eq:relations}  are still satisfied. 

 Also the counterpart of Proposition \ref{prop:quadratic} holds.  We give its proof below.  

\begin{prop}\label{prop:quadratic-t}
There are integers $m_i \in \N, \, m_i\geq n,$ and functions $\varphi _i (u,b,t)$, $\psi_i (u,b,t)$ such that 
\begin{align}\label{eq:generalform-t}
& y_i (u,b, t) = y_i(u,0, t) + b^2 u^{m_i} \varphi_i (u,b, t), \\
\notag
& z_i (u,b, t) = z_i(u,0,t) + b u^{m_i} \psi_i (u,b, t). 
\end{align}
Moreover, either $\varphi_i \equiv\psi _i \equiv  0$ if \eqref {eq:parameterization2-t} parameterizes a branch  of $\Sigma_f$ or 
 $\varphi_i(0,0,0)\ne 0$, $\psi_i(0,0,0)\ne 0$ if \eqref {eq:parameterization2-t} parameterizes a family of polar curves.  
\end{prop}

\begin{proof}
If $y_i(u,b,t)$ and $z_i(u,b,t)$ are independent of $b$ then \eqref{eq:parameterization2-t} parameterizes a branch of the 
singular locus of $\Sigma_f$.  Therefore we suppose that  one of them, and hence by \eqref{eq:relations} both of them,   depend notrivially on $b$.  
 Expand $\frac{\partial z_i}{\partial b} (u,b,t) = \sum_{k\ge m} a_k(b,t) u^k$ 
 with $a_{m}(b,t)\not \equiv 0$.  To show the result it suffices to show that   
  $a_m (0,0)\ne 0$.

Suppose, by contradiction, that $a_m (0,0)=0$.  Then there exists a solution 
$(b(u), t(u))$, with $(b(0), t(0))=0$, of the equation $\frac{\partial z_i}{\partial b} (u,b, t)=0$. 

By the last identity of \eqref{eq:relations}, $(b(u),t(u))$ 
also solves $\frac{\partial y_i}{\partial b} =0$. Recall that $f'_z + bf'_y$ vanishes identically on \eqref{eq:parameterization2}. Thus computing $\frac{\partial }{\partial b} (f'_z + bf'_y)$ on \eqref{eq:parameterization2-t}, and replacing $(u,b,t)$ by 
$(u,b(u), t(u))$ we get 
\begin{align}
0= \frac{\partial }{\partial b} (f'_z + bf'_y)  = 
 (f''_{zy} + bf''_{yy}) \frac{\partial y}{\partial b}  + 
(f''_{zz}  
 + bf''_{yz})\frac{\partial z}{\partial b} + f'_y= f'_y .
\end{align}
Therefore, in this case, \eqref{eq:parameterization2-t} parameterizes a component of $\Sigma_f$.   
\end{proof}

\begin{cor}\label{cor:unitofY}
\begin{align}\label{eq:generalform-Yt}
Y_i (u,b, t) = y_i(u,b, t) - bz_i (u, b, t) = y_i(u,0, t) - bz_i (u,0, t) + b^2 u^{m_i} 
 unit (u,b, t) . 
\end{align}
\end{cor}
\begin{proof}
Using \eqref{eq:generalform-t} we get \\
$Y_i (u,b, t) = y_i(u,b, t) - bz_i (u, b, t) = y_i(u,0, t) - bz_i (u,0, t) + b^2 u^{m_i} (\varphi_i (u,b, t) -\psi_i (u,b, t) ).$  

Differentiating with respect to $b$ and applying 
 \eqref{eq:relations}, we conclude that $(\varphi_i (u,b, t) -\psi_i (u,b, t)) $ is a unit (as $\psi_i$ is unit by \eqref{eq:generalform-t}).

\end{proof}

The following lemma follows from the Zariski equisingularity assumption. 

\begin{lemma}\label{lem:contacts2}
\begin{align}\label{eq:generalformpairs-t}
& y_i(u,b,t)- y_j(u,b,t)  = u^{k_{ij}} unit (u,b,t), \\
\notag
& z_i(u,b,t)- z_j(u,b,t)  = O(u^{k_{ij}}), \\
\notag 
& Y_i(u,b,t)- Y_j(u,b,t)  = u^{k_{ij}} unit (u,b,t), 
\end{align}
and $
y_i(u,b, t) = O(u^n), \quad z_i(u,b,t) = O(u^n)$.
\end{lemma} 

\begin{rmk}\label{rmk:rem-mi-mj-kij}
Note that by Proposition \ref{prop:quadratic-t},  $m_i\neq m_j$ implies $k_{ij}\le \text{min}\{m_i, m_j\}$.
\end{rmk}

\begin{lemma}\label{lem:distance_to_Sigma}
Let $p_i(u,0,t)=(u^n, y_i(u,0,t), z_i(u,0,t))$ parameterize a family of polar curves.    
Then $\dist (p_i(u,0,t) , \Sigma_f) \gtrsim |u|^{m_i}$.  
\end{lemma}

\begin{proof}
Fix a component $\Sigma_r$ of $\Sigma_f$ parameterized by 
$(u^n,\tilde y_r(u,t), \tilde z_r(u,t), t)$.  By Proposition \ref{prop:quadratic} and Zariski equisingularity 
$$
y_i(u,b,t) - \tilde y_r(u,t) = (y_i(u,0,t) - \tilde y_r(u,t)) + u^{m_i} b^2 unit = u^{k_{ir}} unit , 
$$
that is possible only if $m_i\ge k_{ir} \geq n$.  
\end{proof}



\section{Polar wedges} 

In this section we consider the  polar wedges in the sense of 
\cite{BNP2014} and \cite{NPpreprint}.  
The polar wedges are neighbourhoods of the polar curves that play a crucial role in our proof of Theorem 
\ref{thm:maintheorem}.  The formal definition is the following.  

\begin{defi} [Polar wedge]\label{def:polarwedge}
 We call a \emph{polar wedge} and denote it by $\PW_i$  the image of the map $p_i(u,b,t)$ defined by \eqref{eq:parameterization2-t} (for $|b|< \varepsilon$ with $\varepsilon >0$  small), that parameterizes a family of polar curves associated to the projections $\pi_b$. 
\end{defi}

Thus if $p_i(u,b,t)$ of \eqref{eq:parameterization2-t} is independent of $b$, that is it parameterizes a branch of the singular set
$\Sigma_f$, then it does not define a polar wedge. 
Two polar wedges (defined for the same $\varepsilon$) either coincide as sets 
or are disjoint for $u\ne 0$. 
Moreover, either $k_{ij} \le \min\{m_i,m_j\}$ or $k_{ij} > m_i=m_j$.

\subsection {Allowable sectors}\label{ss:allowablesectors}  
Let $\PW_i$ be a polar wedge parameterized by $p_i$  and let $\theta$ be an $n$-th root of unity.  Then $p_i(\theta u,b, t)$ could be identical to $p_i(u,b,t)$ or not, but it always parameterizes the same polar wedge as a set.  In order to avoid confusion 
and also to compare two different polar wedges we work over allowable sectors. 
  We say that a sector $\Xi = \Xi_I= \{u\in \C; \arg u \in I\}$ is \emph{allowable} if the 
 interval $I\subset \R$ is of length strictly smaller than $2\pi/n $.  If we consider only $u\in \Xi$ then 
 $x=u^n\ne 0$ uniquely defines  $u$.  That means that over such an $x,$ every point in the union 
 of polar wedges is attained by a unique parameterization.  

Therefore we may write such parameterization   \eqref{eq:parameterization2-t} in terms of $x,b, t$ assuming implicitly that we work over a sector $\Xi$
\begin{align}\label{eq:wedge-parameterization-x}
p_i(x,b, t) = ( x, y_i(x,b, t), z_i(x,b;t), t) 
\end{align}
with 
\begin{align}\label{eq:generalform-x}
y_i(x,b, t) = y_i(x,0, t) + b^2 x^{m_i/n} \varphi_i (x,b, t) \\
\notag
z _i(x,b, t) = z_i(x,0, t) + b x^{m_i/n} \psi_i (x,b, t). 
\end{align}

\begin{rmk}\label{allowable} We note that any two points in polar wedges 
$p_i (u_1,b_1,t_1)$ and \allowbreak 
$p_j(u_2, b_2, t_2)$ can be compared using parameterizations over the same allowable sector.  Indeed, given nonzero $u_1, u_2$ there always exists  an $n$-th root of unity $\theta$ and an allowable sector $\Xi$  that contains $u_1$ and $\theta u_2$ 
and an index $k$ such that  $p_j(u_2, b_2, t_2)=p_k(\theta u_2, b_2, t_2)$.  
\end{rmk}

\subsection{Distance in polar wedges} 

Having an allowable  sector fixed we show below formulas for the distance between  points inside one polar wedge and 
the distance between  points of different polar wedges.  Note that these formulas imply, in particular, that different 
polar wedges do not intersect outside $T= \{x=y=z=0\}$.  In order to avoid  heavy notation we do not use special symbols
for the restriction of a polar wedge to an allowable sector.  

\begin{prop}\label{prop:distances}
For every polar wedge $\PW_i$  and for $x_1,x_2,b_1,b_2, t_1,t_2$ sufficiently small 
\begin{align}\label{eq:i=j}
\|p_i (x_1,b_1,t_1) -p_i(x_2,b_2,t_2) \| &\sim \max \{|t_1-t_2|, |x_1-x_2|,  |b_1-b_2|  |x_1 |^{m_i/n}  \} \\
\notag 
& \sim \max \{|t_1-t_2|, |x_1-x_2|,  |b_1-b_2|  |x_2 |^{m_i/n}  \} .
\end{align}
For every pair of polar wedges $\PW_i, \PW_j$, if  $k_{ij} \le \min \{m_i, m_j\}$ (in particular if $m_i\ne m_j$) then
\begin{align}\label{eq:ij1}
\|p_i (x_1,b_1,t_1) -p_j(x_2,b_2, t_2) \| & \sim \max \{|t_1-t_2|, |x_1-x_2|,  |x_1 |^{k_{i,j}/n}\} \\
\notag 
& \sim \max \{|t_1-t_2|, |x_1-x_2|,  |x_2 |^{k_{i,j}/n}\}, 
\end{align}
and if $m_i = m_j =m$ then 
\begin{align}\label{eq:ij2}
\|p_i (x_1,b_1, t_1) -p_j(x_2,b_2, t_2) \| & \sim \max \{|t_1-t_2|, |x_1-x_2|,  |x_1 |^{k_{i,j}/n}, |b_1-b_2|  |x_1 |^{m/n}  \} \\
\notag
&\sim \max \{|t_1-t_2|, |x_1-x_2|,  |x_2 |^{k_{i,j}/n}, |b_1-b_2|  |x_2 |^{m/n}   \}.  
\end{align}
\end{prop}

\begin{cor}\label{cor:comparison}
\begin{align*}
& \|p_i(x_1,b_1, t_1) -p_j(x_2,b_2, t_2) \|\\ & \sim \|p_i(x_1,b_1, t_1) - 
p_j(x_1,b_1, t_1) \|  + 
\|p_j(x_1,b_1, t_1) -p_j(x_2,b_2,t_2) \| .
\end{align*}
\end {cor} 

\begin{cor}\label{cor:lipschitz}{\rm [Lipschitz property]}\\ 
There is $c>0$ such that for all $x_1,x_2,b_1,b_2, t$ sufficiently small 
\begin{align*}
  \| p_i(x_1,b_1, 0) -p_j(x_2,b_2,0) \|
  \le c \|p_i(x_1,b_1,t) -p_j( x_2,b_2,t) \| \\
  \le c^2 \|p_i(x_1,b_1,0) -p_j(x_2,b_2, 0)\| . \end{align*}
\end{cor}

\begin{proof}[Proof of Proposition \ref{prop:distances}]
We divide the proof in four steps. In the first two steps we reduce the proofs 
of all \eqref{eq:i=j}, \eqref{eq:ij1}, \eqref{eq:ij2} to simpler cases.  In particular, 
while considering the formula \eqref{eq:i=j} we suppose below that $i=j$.  \\
\textbf{1. First reduction.} \\
We claim that it suffices to prove the formulas \eqref{eq:i=j}, \eqref{eq:ij1}, \eqref{eq:ij2} for $t_1=t_2$.  
This follows from  
\begin{align*}
 \|p_i(x_1,b_1, t_1) -p_j(x_2,b_2, t_2) \| & \sim  |t_1-t_2| + 
  \|p_i(x_1,b_1, t_1) -p_j(x_2,b_2, t_2) \|\\
  &  \sim 
 |t_1-t_2| + \| p_i(x_1,b_1, t_2) -p_j(x_2,b_2, t_2) \|    
\end{align*} 
that we show now.  The first property is obvious, $|t_1-t_2|$ is a part  of 
$ \|p_i(x_1,b_1, t_1) -p_j(x_2,b_2, t_2) \|$.

Secondly, $p_i(x,b, t_1) -p_i(x,b,t_2) = O(t_1-t_2)$ because 
$p_i(u^n,b,t)$ is analytic.  This implies that 
\begin{align*}
&   \|p_i(x_1,b_1, t_1) -p_j(x_2,b_2, t_2) \| \\
&  \le \| p_i(x_1,b_1, t_1) -p_i(x_1,b_1, t_2) \| + \|p_i(x_1,b_1, t_2) -p_j(x_2,b_2, t_2) \| \\
& \lesssim  |t_1-t_2| + \| p_i(x_1,b_1, t_2) -p_j(x_2,b_2, t_2) \| .   
\end{align*} 
A similar computation gives 
$\| p_i(x_1,b_1, t_2) -p_j(x_2,b_2, t_2) \| \lesssim  |t_1-t_2| +  \|p_i(x_1,b_1, t_1) -p_j(x_2,b_2, t_2) \|$. This completes the proof of first reduction claim.   

\noindent
\textbf{2. Second reduction.} \\
We claim that it suffices to show the formulas of the above proposition for the case $t=t_1=t_2, x_1=x_2$.  
The argument is similar to the one above.
The property 
$p_i(x_1 ,b, t) -p_i(x_2,b,t) = O(x_1-x_2)$ follows from 
the following lemma. 

\begin{lemma}\label{lem:x_1-x_2}
We have for each $i$
$$|y_i(u_1,b, t) - y_i(u_2,b, t)| = O(|u_1^n -u_2^n|), \,  |u_1\frac {\partial y_i }{\partial u}  (u_1,b,t) - u_2 \frac {\partial y_i }{\partial u} (u_2,b,t) | = O(|u_1^n -u_2^n|),$$  and similar bounds hold 
for $z_i$ in place of $y_i$.
\end{lemma}

\begin{proof}
If  $(u_1,b, t)$, $(u_2,b, t)$ are in the same allowable sector 
then we have 
\begin{align*}
|u^n_1-u^n_2| \sim |u_1-u_2| \max \{|u_1| ^{n-1}, |u_2|^{n-1}\},  
\end{align*}
that is both sides are comparable up to a constant depending only on the sector. 
Denote $y_i(u,b, t) = u^n \hat y_i(u,b, t)$ and suppose $|u_2|\ge |u_1|$.  Then 
\begin{align*}
& |y_i(u_1,b, t) - y_i(u_2,b, t)| \\
&  \lesssim |(u^n_1-u^n_2)\hat y_i(u_1,b, t)| + 
|u_2^n| |\hat y_i(u_1,b, t) - \hat y_i(u_2,b, t)| \\
& \lesssim |u^n_1-u^n_2| + |u_2^n| |u_1-u_2|\sim |u^n_1-u^n_2| . 
\end{align*}
This shows the first formula; the second one can be shown 
in a similar way. 
 \end{proof} 
\noindent
\textbf{3. Proof of \eqref{eq:i=j} and  \eqref{eq:ij1}.} \\
We assume $t=t_1=t_2, x= x_1=x_2$.  Then  \eqref{eq:i=j} follows from  \eqref{eq:generalform-t}   and the fact that $b\to b\psi(b)$ is bi-Lipschitz ($\psi$ a unit), 
and  \eqref{eq:ij1} follows from 
$$
y_i(x,b_1, t) - y_j(x,b_2, t) = (y_i(x,0,t) - y_j(x,0,t) ) + (b_1^2 x^{m_1/n} \varphi_i(x,b_1,t) - 
b_2^2 x^{m_2/n} \varphi_j (x,b_2,t) )
$$
and a similar formula for $z_i(x,b_1,t) - z_j(x,b_2,t) $.  

\noindent
\textbf{4. Proof of   \eqref{eq:ij2}.} \\
We assume $t=t_1=t_2, x= x_1=x_2$ and  $m=m_1=m_2 $.  Then 
\begin{align}\label{eq:for-y}
y_i(x,b_1, t) - y_j(x,b_2, t)  & = (y_i(x,b_1, t) - y_j(x,b_1, t) ) + (y_j(x,b_1,t) - y_j(x,b_2, t) )  \\
\notag
 & = x^{k_{ij}/n} unit  + x^{m/n} (b_1^2  \varphi_j(x,b_1, t) - 
b_2^2  \varphi_j (x,b_2, t) )  \\
\notag
 & = x^{k_{ij}/n} unit  + x^{m/n} (b_1 -b_2) O(\|(b_1,b_2)\| ). 
\end{align} 
\begin{align}\label{eq:for-z}
z_i(x,b_1, t) - z_j(x,b_2, t)  = O(x^{k_{ij}/n})  + x^{m/n} (b_1 -b_2) (unit  + O(\|(
b_1,b_2)\|) ).
\end{align} 
Now \eqref{eq:ij2} follows from \eqref{eq:for-y}, \eqref{eq:for-z}. Indeed, we may consider separately the three cases:  $ |x |^{k_{i,j}/n} \sim |b_1-b_2|  |x |^{m/n}$, $ |x |^{k_{i,j}/n}$ dominant, and $ |b_1-b_2|  |x |^{m/n}$ dominant,  and suppose that $b_1,b_2$ are small in comparison to the units.    
\end{proof} 



\section{Stratified Lipschitz vector fields on polar wedges}\label{sec:stratlip}

In this section we describe completely the stratified Lipschitz vector 
fields on polar wedges in terms of their parameterizations.  
Note that these descriptions are valid only over allowable sectors, see Remark \ref{allowable}. 

Let $\PW_i$ be a polar wedge parameterized by \eqref{eq:parameterization2-t}.  
We call the polar set $C_i$, parameterized by $p_i(u,t):= p_i(u,0, t)$, \emph{the spine of $\PW_i$}. 
A vector field on 
$\PW_i$ is stratified if it is tangent to the strata: $T$, $C_i\setminus T$, and to $\PW_i\setminus C_i$.

\subsection{Stratified Lipschitz vector fields on a single polar wedge} 

 Let   $p_{i*}(v)$ be a vector field defined on a subset of $\PW_i$, where 
$$
v(u ,b, t) = \alpha (x,b, t)\frac {\partial}{\partial t} + \beta (x,b, t) \frac {\partial}{\partial x} + \delta (x,b, t) \frac {\partial}{\partial b}.\\
$$
We always suppose  the vector field $p_{i*}(v)$ is well defined on $\PW_i$, that is independent of $b$ if $x=0$,
 and it is stratified, that is tangent to $T$ and $C_i\setminus T$;  
 $$p_{i*}(v)=\beta \frac {\partial}{\partial x} +(\beta \frac {\partial y_i}{\partial x} +\delta \frac {\partial y_i}{\partial b} +\alpha \frac {\partial y_i}{\partial t})\frac {\partial}{\partial y} +(\beta \frac {\partial z_i}{\partial x} +\delta \frac {\partial z_i}{\partial b} +\alpha \frac {\partial z_i}{\partial t})\frac {\partial}{\partial z} +\alpha \frac {\partial}{\partial t}.$$

 The independence on $b$ if $x=0$ implies that both $\alpha (0,b,t)$ and $\beta (0,b,t)$ are independent on $b,$ and the actual tangency to $T$ assures that  in fact $\beta (0,b, t)=0$. The tangency to $C_i\setminus T$ implies $\delta (x,0,t)=0$. 
 We also note that 
 $ p_{i*}(\frac {\partial}{\partial b})$ is always zero on $x=0$.

Suppose that a function $h (u,b, t)$ defines a function $\tilde h = h \circ p_i^{-1}$ on $\PW_i$,
 that is $h (0,b, t)$ does not depend on $b$.  Then, after Proposition \ref{prop:distances}, $\tilde h$ is Lipschitz iff 
\begin{align}\label{eq:wedgelipschitz}
|h(u_1,b_1, t_1) -h(u_2,b_2,t_2) | \lesssim |t_1-t_2| + |u^n_1- u^n_2| +  |b_1-b_2|  |u_2 |^{m}   .
\end{align}

\begin{prop} \label{prop:LVF1}
The vector fields $p_{i*} (\frac {\partial}{\partial t})$, $p_{i*} ( u\frac {\partial}{\partial u})$, 
$p_{i*} (b \frac {\partial}{\partial b})$  are stratified Lipschitz on $\PW_i$. \\
\end{prop} 

\begin{proof}
We show that each coordinate  of these vector fields is Lipschitz.  For this computation it is more convenient  to use the parameter $u$ instead of $x$ since these vector fields are analytic in $u,b, t$.  For clarity we also drop the index 
$i$ coming from the parameterization \eqref{eq:parameterization2-t}.  

The $t$-coordinate of  $p_{*} (\frac {\partial}{\partial t})$ equals $1=\frac {\partial t}{\partial t}$ and is Lipschitz.  The $x$-coordinate of  $p_{*} (\frac {\partial}{\partial t})$ vanishes identically.  Let us show, using Proposition \ref{prop:quadratic-t} and  Lemma \ref{lem:x_1-x_2}, that 
the $y$-coordinate of  $p_{*} (\frac {\partial}{\partial t})$  is Lipschitz (the argument for the $z$ coordinate  
is similar)  
\begin{align*}
& |\frac {\partial y }{\partial t}  (u_1,b_1, t_1) -  \frac {\partial y }{\partial t} (u_2,b_2,t_2) | \\
& \le |\frac {\partial y }{\partial t}  (u_1,b_1, t_1) -  \frac {\partial y }{\partial t} (u_1,b_1,t_2) | + |\frac {\partial y }{\partial t}  (u_1,b_1,  t_2) -  \frac {\partial y }{\partial t} (u_2,b_1,t_2) | \\ & + |\frac {\partial y }{\partial t}  
(u_2,b_1,t_2) -  \frac {\partial y }{\partial t} (u_2,b_2,t_2) | 
  \lesssim |t_1-t_2| + |u^n_1-u^n_2| +  |b_1-b_2|  |u_2 |^{m} \\&
   \sim \max \{ |t_1-t_2|, |u^n_1-u^n_2|, |b_1-b_2|  |u_2 |^{m}\}.
\end{align*}
A similar computation works for $p_*( x\frac {\partial}{\partial x}) = \frac 1 n p_*( u\frac {\partial}{\partial u})$
\begin{align*}
& |u_1 \frac {\partial y }{\partial u}  (u_1,b_1,t_1) -  u_2 \frac {\partial y }{\partial u} (u_2,b_2,t_2) | \\
& \le |u_1 \frac {\partial y }{\partial u}  (u_1,b_1,t_1) - u_1  \frac {\partial y }{\partial u} (u_1,b_1,t_2) | + |u_1 \frac {\partial y }{\partial u}  (u_1,b_1,t_2) - u_2  \frac {\partial y }{\partial u} (u_2,b_1,t_2) | \\ & + |u_2 \frac {\partial y }{\partial u} 
 (u_2,b_1, t_2) -  u_2 \frac {\partial y }{\partial u} (u_2,b_2,t_2) |   \lesssim |t_1-t_2| + |u_1^n-u_2^n| +  |b_1-b_2|  |u_2 |^{m} \\ & \sim \max \{ |t_1-t_2|, |u^n_1-u^n_2|, |b_1-b_2|  |u_2 |^{m}\} .
\end{align*}
All the other cases can be checked in a similar way.  
\end{proof}

\begin{prop}\label{prop:LVFcriterion1}
The vector field of the form $p_{i*}(v)$, defined on a subset 
$U$ of $\PW_i$ containing $C_i$, 
 is stratified Lipschitz iff the following conditions are sati\-sfied: \\
1) $ \alpha $ satisfies \eqref{eq:wedgelipschitz};\\
2) $|\beta |\lesssim |x|$ and 
$ \beta $ satisfies \eqref{eq:wedgelipschitz};\\
3) $|\delta |\lesssim |b|$ and 
$ \delta x^{m/n}$  satisfies \eqref{eq:wedgelipschitz}.  
\end{prop}

\begin{proof}
If $p_{i*}(v)$ is Lipschitz then so is its $t$-coordinate, that is $\alpha$.  
We claim that if $\alpha$ satisfies \eqref{eq:wedgelipschitz} so do $\alpha \frac {\partial y_i }{\partial t} $ and 
$\alpha \frac {\partial z_i }{\partial t} $.  This follows from  Proposition \ref{prop:LVF1} because the product of two Lipschitz functions is Lipschitz.  This shows that $p_ {i*} (\alpha \frac {\partial  }{\partial t} )$ is Lipschitz.  
By subtracting it from $p_{i*}(v)$ we may assume that $\alpha\equiv 0$.  

If $p_{i*}(v)$ is Lipschitz then so is its $x$-coordinate, that is $\beta $. 
Let $(x,b,t) \in p_i^{-1}(U)$.  Then, by \eqref{eq:i=j} in Proposition \ref{prop:distances} and the Lipschitz property between $p_i(x,b,t)$ and $p_i(0,b,t)$,  we have  $|\beta |\lesssim |x|$ as claimed.

To use a similar argument to the previous "the product of Lipschitz functions is Lipschitz", we need the following elementary generalization.

\begin{lemma}\label{lem:funny}
Suppose $ h:X\to \C$ is a Lipschitz function on a metric space $X$ and let
$L_{h}:=\{ f:X\to \C;  \text{ Lipschitz on X }, |f| \lesssim |h| \}$. If $f, g \in L_{h}$,  then $\xi := fg/h \in L_{h}$ (here $\xi$ is understood to be equal to $0$ on the zero set of $h$).
\end{lemma}

\begin{proof} 
Suppose $|h(q_2)|\ge |h(q_1)|$.  
Then $| fg(q_2) - fg(q_1)|  \lesssim  |h(q_2)|  \dist (q_1,q_2) $ and
\begin{align*}
|\xi (q_2) -\xi (q_1)| &\le \frac{| fg(q_2) h (q_1) - fg(q_1) h (q_2)| }{|h (q_1) h (q_2)| } \\
& \le \frac{| fg(q_2) h (q_1) - fg(q_1) h (q_1)| 
+| fg(q_1) h (q_1) - fg(q_1) h (q_2)| }{|h (q_1) h (q_2)|} \\
& \lesssim \dist (q_1,q_2) . 
\end{align*}
\end{proof}

We apply the above lemma to $ f=\beta$, $ g=p_{i*}( x\frac {\partial}{\partial x})$, 
and $h= x$ respectively,
to complete the proof of  2).    
Thus, by subtracting $p_{i*}( \beta \frac {\partial}{\partial x})$ from $p_{i*}(v)$ we may assume that $\beta \equiv 0$.

Consider now $p_{i*}(\delta \frac {\partial  }{\partial b} ) = 
(0, \delta \frac {\partial  y_i}{\partial b} , \delta \frac {\partial z_i }{\partial b}, 0 )$.  
By Proposition \ref{prop:LVF1},  $p_{i*}(b\frac {\partial  }{\partial b} )$ is Lipschitz and by  \eqref{eq:generalform-t} it satisfies $\|p_{i*}(b \frac {\partial  }{\partial b} )\| \lesssim |b| |x^{m/n}| $.  Therefore if $ \delta x^{m/n}$  satisfies \eqref{eq:wedgelipschitz} then  
$p_{i*}(\delta \frac {\partial  }{\partial b} )$ is 
Lipschitz if we apply Lemma \ref{lem:funny}  to  $f=\delta x^{m/n}$, 
$ g=p_{i*}(b \frac {\partial  }{\partial b} )$, and 
$h=  b x^{m/n}$. 

Conversely, if $p_{i*}(\delta \frac {\partial  }{\partial b} )$ is Lipschitz  so is its $z$-coordinate 
$ \delta \frac {\partial  z_i}{\partial b} $. Moreover, because 
$p_{i*}(\delta \frac {\partial  }{\partial b} )$ is stratified (tangent to 
$C_i$), $ \delta \frac {\partial  z_i}{\partial b} $ is zero if $b=0$.  
Therefore, since $ \frac {\partial  z_i}{\partial b} \sim x^{m/n}$ by 
\eqref{eq:generalform-t} and by (21) in Proposition \ref{prop:distances} and the Lipschitz property between $p_i(x,0,t)$ and $p_i(x,b,t)$,  we have  
$|\delta |\lesssim |b|$.  By  Lemma \ref{lem:funny}  applied to   
$f= \delta \frac {\partial  z_i} {\partial b} $, $g= b 
x^{m/n} $  and 
$h= b  \frac {\partial  z_i}{\partial b}$, we conclude that 
$\delta x^{m/n}$ satisfies \eqref{eq:wedgelipschitz}.   
\end{proof}


\subsection{Lipschitz vector fields on the union of two polar wedges}

Consider two polar wedges $\PW_i$ and $\PW_j$ parameterized by $p_ i (x,b,t) $ and $p_j(x,b,t) ,$ over the same allowable sector, see \ref{ss:allowablesectors} for more details.   

Let $\tilde h$ be a function defined on a subset of $\PW_{i} \cup \PW_{j}$ by two functions  $h_k (x,b,t)$, $k=i,j$. 
 Then, after  Proposition \ref{prop:distances}, $\tilde h$ is Lipschitz iff so are its restrictions $\tilde h_i$ and 
 $\tilde h_j$ to 
 $\PW_i$ and $\PW_j$ respectively, 
 and
\begin{align}\label{eq:twowedgelipschitz}
|h_i(x_1,b_1,t_1) -h_j(x_2,b_2,t_2) | \lesssim |t_1-t_2| + |x_1-x_2| + |x_2|^{k_{ij}/n}+ |b_1-b_2|  |x_2 |^{m/n}   ,
\end{align}
where $m= \min \{m_i, m_j\}$.

\begin{prop} \label{prop:LVF2}
The vector fields given by  $p_{k*}(v)$, $k=i,j$, where $v$ are $\frac {\partial}{\partial t}$, $x \frac {\partial}{\partial x}$, or $b\frac {\partial}{\partial b}$,  are Lipschitz on  $\PW_{i} \cup \PW_{j}$.  
\end{prop} 

\begin{proof}
By Corollary \ref{cor:comparison} and Propostion \ref{prop:LVF1} it suffices to check only the condition \eqref{eq:twowedgelipschitz} for 
$t=t_1=t_2$, $u=u_1=u_2$, and $b= b_1=b_2$.  In this case the result follows the facts  that $\|p_{i} - p_j \| \lesssim u^{k_{ij}}$ and that 
$(p_{i} - p_{j})(u,b,t) = u^{k_{ij}} q (u,b,t)$, with $q$ analytic,  see Lemma \ref{lem:contacts2}.  
\end{proof}

For $k=i,j$ let   $p_{k*}(v_k)$ be a vector field on a subset of $W_{\Xi, k}$ given by 
 $$
v_k(x,b;t)= \alpha_k \frac {\partial}{\partial t} + \beta_k \frac {\partial}{\partial x} + \delta_k \frac {\partial}{\partial b} .
$$

\begin{prop}\label{prop:LVFcriterion2}
The vector field given by  $p_{k*}(v_k)$, $k=i,j$, defined on a subset 
$U$ of $\PW_{i} \cup \PW_{j}$ containing $C_i\cup C_j$,  is stratified Lipschitz  iff the following conditions are satisfied: \\
 0) each $p_{k*}(v_k)$ is stratified Lipschitz on $U\cap \PW_{k}$; \\ 
1) $ \alpha _i, \alpha_j $ satisfy \eqref{eq:twowedgelipschitz};\\
2) 
$ \beta_i, \beta_ j$  satisfy \eqref{eq:twowedgelipschitz};\\
3) 
$ \delta_i x^{m/n},  \delta_j x^{m/n}$  satisfy \eqref{eq:twowedgelipschitz}.   
\end{prop}

\begin{proof}
The proof is similar to the proof of Proposition \ref{prop:LVFcriterion1} and it is based on   Lemma \ref{lem:funny} and 
Proposition \ref{prop:LVF2}.   
\end{proof}

\begin{rmk}
If $\tilde h_i$, $\tilde h_j$ are stratified Lipschitz on 
 $\PW_i$ and $\PW_j$ respectively, then, by Corollary \ref{cor:comparison}, it suffices to check \eqref{eq:twowedgelipschitz}
 for $t=t_1=t_2$, $u=u_1=u_2$, and $b= b_1=b_2$.  Therefore, in this case, 
 \eqref{eq:twowedgelipschitz} can be replaced by 
 \begin{align}\label{eq:twowedgelipschitzprime}
|h_i(x,b,t) -h_j(x,b,t) | \lesssim |x|^{k_{ij}/n}.
\end{align}
\end{rmk}



\section{Proof of Theorem \ref{thm:maintheorem}.  Part I}\label{sec:proof-partI}

We show the statement of Theorem \ref{thm:maintheorem} on $\PW $, that is 
the union  of the polar wedges and the singular set $\Sigma_f$.  

\subsection{Extension of stratified Lipschitz vector fields on polar wedges in the non parameterized case} \label{ss:lip-strat-nonpar}

Let $X=\{ f(x,y,z)=0\}$,  $S=\{f(x,y,z)=f'_z(x,y,z)=0\}$, 
and $f$ satisfies the Transversality Assumptions.  
We show that $\{\PW\setminus S, S\setminus \{0\}, \{0\} \}$ is 
a Lipschitz stratification of $\PW $  
in the sense of Mostowski.  

Given $q_0\in S\setminus \{0\}$ and a vector $v_0= v(q_0)$ tangent to $S$.  
Suppose $q_0$ belongs to a component $S_i$ (a polar curve or a branch of the singular locus) 
of $S$ parameterized by 
$$
p_i(x) = (x , y_i(x), z_i(x)), \quad q_0= p_i(x_0)  
$$
and $v_0 = p_{i*} (\beta_0 \frac \partial {\partial x}) $.  
Then the vector field on $S$ defined on each $S_j$ by $v_j= p_{j*} (\beta  x\frac \partial {\partial x}) $, with $\beta  = \beta_0/x_0$, defines a Lipschitz extension of $v_0 $.  This shows (L1). 

Consider a stratified Lipschitz vector field $v$ on $S\cup\{q_0\}$ with $q_0= p_i(x_0,b_0) \in \PW_i$ defined by ${p_j}_* v_j$ on the  component $S_j$ of $S$, where 
 $$
v_j(x,b)=  \beta_j \frac {\partial}{\partial x} + \delta_j \frac {\partial}{\partial b} .  
$$
 Thus, for $j\ne i$, the functions $\beta_j$ and $\delta_j$ are defined only for $b=0$ 
 (and hence $\delta_j=0$ since the vector field is stratified). 
  The functions  $\beta_i$ and $\delta_i$ 
 are defined on 
$\{(x,b); b=0\} \cup \{(x_0,b_0)\}$. 
Denote $\beta_0= \beta_i (x_0,b_0)$,    $\delta_0= \delta_i  (x_0,b_0)$.  
By Propositions \ref{prop:LVFcriterion1} and \ref{prop:LVFcriterion2} 
it suffices to extend $\beta_j$ and $\delta_j$ to two families of functions, still denoted by 
 $\beta_j$, $\delta_j$, that satisfy the conditions given in those 
propositions.  
 For all $j$ we define
 \begin{align}\label{eq:extensionbeta}
 &\beta_j (x,b) =  (\beta_0   - \beta_i (x_0,0))\frac b {b_0} 
 \frac {x^{m_j/n}} {{x_0}^{m_i/n}} + \beta_j (x,0),  \\
 \label{eq:extensiongamma}
 &\delta_j  (x,b) = (\delta_0 b)/b_0 .
 \end{align}
Then, because $ |\beta_0   - \beta_i (x_0,0)| \le CL |b_0| |x_0|^ {m_i/n} $, where $L$ is the Lipschitz constant of the vector filed $v$ and $C$ is a universal constant, the first summand of the right-hand side of 
\eqref{eq:extensionbeta} satisfies 2) of Propositions \ref{prop:LVFcriterion1} and \ref{prop:LVFcriterion2}.  
The argument for \eqref{eq:extensiongamma} is similar because $|\delta_0|\le CL |b_0|$.
 This completes the proof of Theorem \ref{thm:maintheorem} for $\PW $ 
 in the non-parameterized case.  \\

\subsection{Parameterized case}
By Corollary \ref{cor:lipschitz} and Propositions \ref{prop:LVFcriterion1}, \ref{prop:LVFcriterion2}, the map given $\mathcal X_0 \times T \to \mathcal X$, 
restricted to $\PW  \cap \mathcal X_0$, 
defined in terms of the parameterizations of polar wedges by 
$$
(p_i(0,x,b),t) \to p_i(x,b,t), 
$$
is not only Lipschitz but also establishes a bijection between the
Lipschitz vector fields. Therefore, by Proposition 
\ref{prop:LipStr-charcterization},   
$\{\PW  \setminus S, S\setminus T, T\}$ is a Lipschitz stratification if and only if so is its intersection with $\mathcal X_0$ and the latter is a Lipschitz stratification by the non-parameterized case. 
We use here an easy observation that the cartesian product of a Lipschitz stratification by a smooth space is also Lipschitz (actually the cartesian product of two 
Lipschitz stratifications is Lipschitz).

\subsection{Examples}
In \cite{mostowski88} Mostowski gives a criterion for the codimension 1 stratum of Lipschitz stratification.  In particular he proposes 
the following example (we change the order of variables so it follows our notation): $f(x,y,z)= z^2- (y^3 + y^2x^2)$. The singular set $\Sigma_f $ of $X=\{f=0\}$ is the $x$-axis but as Mostowski shows 
$\{X \setminus \Sigma_f \,, \Sigma_f \setminus \{0\}, \{0\} \}$ is not a Lipschitz stratification of $X$.  
 By solving the system $f= \partial f/\partial z - b \partial f/\partial y=0$ one can check that there is one polar wedge with $n=1$ and $m=4$ given by 
$$
y= - x^ 2 + b^2x^4 \varphi(x,b), z = 3bx^4 \psi(x,b), 
$$
and one has to add a generic polar curve, or just a curve $y= - x^ 2 + b^2x^4+\cdots, z = 3bx^4+\cdots$, to $\Sigma_f$ to get the one dimensional stratum.  
In \cite[Section 7]{mostowski88} Mostowski studies the case of surface singularities in $\C^3$ and shows in particular the following result.

\begin{prop}\label{prop:notLipschitz}
If $X$ has isolated singularity but  there is an $m_i>n$  then $\{X \setminus  \{0\}, \{0\} \}$ is not 
a Lipschitz stratification of $X$.   
\end{prop}

We give below an alternative proof of this proposition.

\begin{proof}
Let $q_0=p(x_0,b_0)\in X \setminus  \{0\}$ be on the polar wedge parameterized by 
$p(x,b)=(x,y(x,b),z(x,b)) $, $x=u^n$, where $y,z$ are as  in  \eqref{eq:generalform}.  Let  $v_0 = p_* ( \frac {\partial}{\partial b})$ be the vector tangent  at $q_0=p(x_0,b_0)$ to $X$.  We extend it by $0$ to 
$\{0\}$ and get a Lipschitz vector field on $\{0\}\cup \{q_0\}$ with Lipschitz constant 
$L=C  x_0^{m/n-1}$, 
where $C>0$ depends only on the polar wedge.   Suppose we extend this vector field 
to $q_1= p(x_1,b_1))$, $x_0=x_1$, by $v_1 = p_* (\alpha_1 \frac {\partial}{\partial x} +  \delta_1 
\frac {\partial}{\partial b}) $ so  that the extended vector field has Lipschitz constant  
$L_1 = C_1  L$.  By  the Lipschitz property of the $x$-coordinate of this vector field 
$|\alpha_1|\le C_1L \|q_0-q_1\| \sim C_1L|b_0-b_1| |x_0|^ {m/n}$. 
Therefore, we can subtract from $v_1$ the vector 
$p_* (\alpha_1 \frac {\partial}{\partial x})$ without changing significantly the Lipschitz constant (just changing $C_1$).  Thus we may assume that $\alpha_1=0$.
  By  the Lipschitz property of the $y$ and $z$-coordinates of this vector field 
\begin{align}\label{eq:Cramer}
& b_0 x_0^{m/n} \tilde \varphi (x_0,b_0) - \delta_1 b_1 x_0^{m/n} \tilde \varphi (x_0,b_1) =   O(|b_0-b_1| x_0^{m/n} ) L_1,   \\ \notag
&  x_0^{m/n} \tilde \psi (x_0,b_0) - \delta_1  x_0^{m/n} \tilde \psi (x_0,b_1) = 
O(|b_0-b_1|x_0^{m/n})L_1 ,
\end{align} 
where $\tilde \varphi, \tilde \psi$ are units.  
Considering \eqref{eq:Cramer} as a system of linear equations with the unknowns $1$ (in front of the first summands of both equations) and $\delta_ 1$,   by Cramer's rule, 
$$1 \lesssim |L_1|\sim |x_0^{m/n-1}|, \quad |\delta_1|  \lesssim |L_1|\sim |x_0^{m/n-1}|,
$$
that is impossible if we allow $x_0\to 0$, as by our assumption $m>n$.
\end{proof}



\section{Quasi-wings} 

Quasi-wings were introduced by Mostowski in \cite[Section 5]{mostowski85} in order to show the existence of Lipschitz stratification in complex analytic case.  
In this construction Mostowski used several co-rank one projections, 
instead of a single one, to cover the whole complement of $\Sigma_f$ 
in $\mathcal X$ by quasi-wings.  We use the quasi-wings to study Lipschitz vector fields on the complement of $\PW $.  \\

The main idea of construction goes as follows (the details will follow later).  Given a real analytic arc $p (s)$, $s\in [0,\varepsilon )$,
of the form 
\begin{align}\label{eq:gammaform}
p(s) =(s^n, y(s), z(s), t(s)), \qquad y(s) = O(s^n), z(s)=O(s^n) .
\end{align}
Our goal is to embed $p(s)$ in a quasi-wing $\QW$ (kind of cuspidal neighborhood of $p(s) $ in $\mathcal X$), that is the graph of a root of $f$ over a set $\W_q$,   
the image of 
\begin{align*}
q(u, v,t) = ( u^n, y(u,t) + u^{\tilde l} v,t), 
\end{align*}
where $u, v\in \C$  are supposed small. Geometrically, $\W_q$ is 
 a cuspidal neighborhood of  $\pi (p(s))$, that we call a wedge,  and 
 $\QW$  is its lift to $\mathcal X$.  
 Thus $\QW$ admits a parameterization of the form  
$p(u, v,t)= (q(u, v,t), z(u,v,t))$ such that $p(s) = p(s, 0,t(s))$. 
We shall make the following assumptions on  $p(s)$ : 
\begin{enumerate} 
\item
	$p(s)$ is not included in $S$ and moreover 
for every polar branch $C_i$ there is an exponent $l_i$ such that 
$s^{l_i}  \sim \dist (p (s), C_i) \sim \dist (\pi(p (s)), \pi (C_i))$.  
A similar assumption is made on every branch of the singular locus 
$\Sigma_f$. 
  In particular  we have $ \dist (p(s), S) \sim \dist (\pi(p (s)), \pi (S))$.  
\item
for every polar branch $C_i$ we have $l_i\le m_i$
(For the definition of $m_i$ see Proposition \ref{prop:quadratic-t}.) This implies that $p(s)$ is not included in $\PW_i$.     
\end{enumerate}
We have  the following requirement on $\QW :$ 
\begin{enumerate}
	\item [(3)]
$
s^{\tilde l}  \lesssim \dist (p (s), S) \sim \dist (\pi(p (s)), \pi (S)),
$\\
that is $\QW$ does not touch $S$ (except along $T$), 
and this property is preserved by the projection to the $t,x,y$-space. 
\end{enumerate}
Then $\PW \cap \QW$ is just the $T$ stratum 
and as we show in  Proposition \ref{prop:deriv-bounded} 
\begin{enumerate}
	\item [(4)]
 $\QW$ 
is the graph of a root of $f$ whose all first order partial derivatives are 
bounded.
In particular, the projection $\QW\to \W_q$ is bi-Lipschitz.   
\end{enumerate} 
In the formal definition of quasi-wings we will require that 
$\tilde l$ is chosen minimal for (3), i.e. 
$
s^{\tilde l}  \sim \dist (p (s), S) \sim \dist (\pi(p (s)), \pi (S)),
$
(we seek the maximal possible set $\W_q$  satisfying the above properties). 
We show in Proposition \ref{prop:existenceq-wing} that each real analytic arc satisfying (1) and (2) can be embedded in a quasi-wing.  
In general, any real analytic arc that is not embedded in the singular locus, satisfies the conditions (1) or (2) after a small linear change of coordinates and therefore can be embedded in a quasi-wing in this new system of coordinates, see Corollary \ref{cor:existenceq-wing}. 
We note that our construction of quasi-wings differs 
significantly from 
the one of Mostowski.   We use the Puiseux with parameter theorem 
and arc-wise analytic trivializations of \cite{PP17}.  The latter also provides 
a crucial partial Lipschitz property, see Remark \ref{rmk:partialLipschitz} that we use in the proof of Proposition \ref{prop:existenceq-wing}. 
Consequently, our construction of quasi-wings can be  extended to the real analytic set-up. 
 Mostowski uses instead the bound on derivatives of holomorphic functions (Schwarz's Lemma).


\subsection{Regular wedges and quasi-wings}

Let $\Delta (x,y,t)$ denote the discriminant of  $f(x,y,z,t)$.  The discriminant locus $\Delta = 0$ is the finite union of families of analytic plane curves
 parameterized by 
\begin{align}\label{eq:discrbranch}
(u,t) \to (u^n, y_i(u,t),t) .
\end{align}
By the Zariski equisingularity assumption we have 
$$y_i(u,t)- y_j(u,t) = u^{k_{ij}} unit (u,t),$$ and by the Transversality Assumptions  
$y_i(u,t) = O(u^n)$.  
Note that $y_i$ of \eqref{eq:discrbranch} is either the projection of a polar branch, 
the one denoted by $y_i(u,0,t)$ in \eqref{eq:generalform-t} and from now on it will be indexed by  $i\in I_C$, or parameterizes the projection of a branch of the singular locus $\Sigma_f$,  and it will be indexed by  $i\in I_\Sigma.$  

Given analytic family of analytic arcs 
\begin{align}\label{param-q}
q(u,t) = (u^n, y(u,t),t) . 
\end{align}
We assume $y(u,t) = O(u^n)$ and that for each discriminant branch \eqref{eq:discrbranch}, $y(u,t)$ satisfies, for some integers $\tilde l_i$,
$$
y(u,t) - y_i(u,t) = u^{\tilde l_i} unit (u,t).
$$
\begin{rmk}
As both $y(u,t) = O(u^n)$ and $y_i(u,t) = O(u^n)$ it follows that $\tilde l_i\geq n.$
\end{rmk}

Consider 
the map
\begin{align}\label{eq:param-q}
q(u, v,t) = ( u^n, y(u,t) + u^{\tilde{l}} v,t),
\end{align}
defined for complex $v$, $|v|< \varepsilon$ with $\varepsilon >0$ small, {and denote its image by $\horn_q$.  
We suppose $\tilde l\ge  \max_i \tilde l_i$, that is the image of $q$, for $u\ne 0$,  is inside the complement of the discriminant locus 
$\Delta=0$.  

\begin{lemma}\label{lem:discriminant}
Let $ g (u,v,z,t) = f(q(u,v,t), z)$.  If $\tilde l\ge  \max_i \tilde l_i$ then the discriminant of $g$ satisfies
\begin{align}\label{eq:discr-g2}
\Delta_g = u^N unit(u,v,t) . 
\end{align}
\end{lemma}

\begin{proof}
Write the discriminant of $f$
$$
\Delta (u^n, y, t) = unit (u,y,t) \prod_i (y-y_i (u,t))^{d_i}.
$$
Then, by assumption $\tilde l\ge  \max_i \tilde l_i$,
$$
\Delta_g (u,v,t) = \Delta (u^n, y (u,t) + v u^{\tilde l}, t) = 
 \, u^{\sum \tilde l_i d_i} unit (u,v,t) .
$$
 \end{proof}

Therefore, by Puiseux with parameter theorem, after a ramification in $u$, we may assume that the roots of $g$ are analytic functions of the form 
$z_\tau (u,v,t) = {z}_\tau ( u^n, y(u,t) + v u^{\tilde l} ,t)$ and that for every pair of such roots   
\begin{align}\label{eq:generalformpairs2}
(z_\tau (u,v, t)- z_\nu (u,v, t))\sim u^{r_{\tau \nu}} .
\end{align}

Moreover, by transversality of projection $\pi$, $z_\tau (u,v,t) = O(u^n)$. 

\begin{prop}\label{prop:deriv-bounded}
Suppose $\tilde l_i\le m_i$ for every projection \eqref{eq:discrbranch} of  a polar  branch.  Then the (first order) partial derivatives of the roots 
$z_\tau (x,y,t)$  of $f$ over $\horn_q$ (the image of \eqref{eq:param-q}), are bounded. 
Therefore, the roots of $g$ are of the form 
\begin{align}\label{eq-qwroot}
z_\tau(u,v,t) = z_\tau(u,t) +vu^{\tilde l} \tilde \psi (u,v,t),
\end{align}
with $\tilde \psi (u,v,t)$ analytic.  
\end{prop}

\begin{proof}  
The derivative $\frac {\partial}{\partial t} (z_\tau (x,y;t))$ is bounded on $\horn_q$ because 
$z_\tau (u,v;t)$ is analytic in $t$.  Similarly $x \frac {\partial}{\partial x} (z_\tau (x,y;t))$ is  $O(x)$ because $z_\tau (u,v;t)$ is analytic in $u$ and 
$$x\frac {\partial z_\tau}{\partial x}  \simeq u\frac {\partial z_\tau}{\partial u}  \lesssim u^{n}.$$  
Finally, $\frac {\partial}{\partial y} (z_\tau (x,y,t))$ is bounded on $\horn_q$ by the conditions $\tilde l_i\le m_i$, 
$\tilde l_i\le \tilde l$, and \eqref{eq:generalform-t}.  
Indeed, since 
$f(x,y,z_\tau (x,y,t),t) \equiv 0 $ we have on the graph of $z_\tau$ 
$$
0 = \frac {\partial }{\partial y} f(x,y,z_\tau (x,y,t),t) = f'_y + \frac {\partial z_\tau}{\partial y}  f'_z.  
$$
If $|\frac {\partial z_\tau }{\partial y} |> N$, then, by \eqref{eq:F'_z}, the graph of $z_\tau (x,y,t)$ on $\horn_q$ would 
intersect a polar wedge $\PW_i$ for 
$b= (\frac {\partial z_\tau }{\partial y} )^{-1}$.
This is only possible if $\tilde l_i\ge \min\{\tilde l, m_i\}$.  If $\tilde l_i= \min\{\tilde l, m_i\}$ then  this intersection is empty provided  we suppose both $b$ and $v$ sufficiently small (and hence $N$ large).  
\end{proof}

We introduce now a version of quasi-wings and nicely-situated quasi-wings of 
\cite{mostowski85}. 

\begin{defi}[Quasi-wings]\label{def:qwings}
We say that the image of $q(u, v, t)$ of \eqref{eq:param-q} is \emph{a regular wedge 
$\horn_q$} if  
$\tilde l=  \max_{i\in I_C\cup I_\Sigma} \tilde l_i$ and  
if $\tilde l_i\le m_i$ for every $i\in I_C$.  Then by 
\emph{a quasi-wing $\QW_\tau$} over $\horn_q$ we mean the image of an analytic map 
$p_\tau (u,v, t)= (q(u, v,t), z_\tau (u,v,t))$, where $z_\tau$ is a root of 
$f(q_t(u,v), z)$.  

We say that two quasi-wings $\QW_\tau, \QW_\nu$ are 
\emph{nicely-situated} if they lie over the same regular wedge $\horn_q$. 
\end{defi}


\subsection{Construction of quasi-wings}
Consider a real analytic arc $p(s)$, $s\in [0,\varepsilon )$, of the form 
\begin{align}\label{A1}
& p(s) =(s^n, y(s), z(s), t(s)), \,  \pi(p(s))=q(s)=(s^n, y(s), t(s)), \\ \notag
&  y(s) = O(s^n), z(s)=O(s^n).
\end{align}
Under some additional assumptions we construct 
in Proposition \ref{prop:existenceq-wing} a quasi-wing containg the arc $p(s)$. 
For this we use in the proof of Lemma \ref{lem:wing} the arc-wise analytic trivializations of \cite{PP17} and construct, following \cite[Proposition 7.3]{PP17},  of a complex analytic wing containing $q(s)$. 

Let $$(u^n, y_i(u,t), z_i(u,t) ,t), \, i\in I_C,$$ be a parameterization of the polar branch $C_i$, and let $$(u^n, y_k(u,t), z_k(u,t), t), \, k\in I_\Sigma,$$
 be a parameterization of the branch $\Sigma_k$ of the singular set $\Sigma_f$. 
\begin{lemma}\label{lem:wing}
Let $q(s) = (s^n, y(s), t(s))$, $y(s) = O(s^n)$,  be a real analytic arc at the origin.  
For each polar branch $C_i$, parameterized as above, 
denote $q_i(u,t) = (u^n, y_i(u,t), t)$ and let  $\tilde l_i = \ord_s (y(s) - y_i(s,t(s))$.  
Then there is a complex analytic wing parameterized by 
$$
q(u,t) = (u^n, y(u,t), t), \quad y(u,t) = O(u^n)
$$
containing $q(s)$, that is satisfying $y(s)= y(s, t(s))$,  such that $y(u,t) - y_i(u,t)$ equals 
 $u^{\tilde l_i}$ times a unit. In particular, over  the same allowable sector we have 
\begin{align}\label{eq:distancewingpolar}
\|(u_1^n, y(u_1,t_1) ,t_1) -(u^n_2,y_i(u_2,t_2),t_2) \| 
\sim \max \{|t_1-t_2|, |u^n_1-u^n_2|,   |u_2 |^{\tilde l_i}  \} 
\end{align}
and $ \ord_s \dist (q (s), \pi(C_i))= \tilde l_i$.  
\end{lemma}

\begin{proof}

By \cite[Theorem 3.3]{PP17} there is an arc-wise analytic local trivialization 
$\Phi: \C^2\times T \to \C^2\times T$ preserving the discriminant locus $\Delta =0$.  
In particular, $\Phi$ is of the form 
\begin{align}\label{eq:arcwise-trivialization}
\Ph (x,y, t) = (\Psi_1 (x,t), \Psi_2 (x,y,t), t) ,
\end{align}
is complex analytic with respect to $t$, and both $\Phi$ and its inverse 
$\Ph^{-1}$ are real analytic on real analytic arcs. 
By \cite[Proposition 3.7]{PP17} we may require $\Psi_1(x,t)=x$, so the allowable sectors are preserved. 

By the arc-analyticity of $\Ph^{-1}$, there exists a real analytic arc 
$(s^n, \tilde y(s), t(s))$ such that  
$ \Ph (s^n, \tilde y(s), t(s)) = (s^n,  y(s), t(s))$.  Then, by the arc-wise analyticity of 
$ \Ph,$ the map 
$q(s,t) = \Ph (s^n, \tilde y(s), t)$ is analytic in both $s$ and $t$, and 
its complexification $q(u,t)$ is a complex analytic wing containing $q(s)$.

\begin{rmk}\label{rmk:partialLipschitz}
Arc-wise analytic trivializations of \cite{PP17} satisfy 
\emph{a partial Lipschitz property}, namely they are bi-Lipschitz for the last variable, i.e., $\Psi_1$ with respect to $x$ and $\Psi_2$ with respect to $y$, etc., see  
\cite[property (Z3) of Theorem 3.3]{PP17}.  
\end{rmk}

By the partial Lipschitz property   
$$
s^{\tilde l_i} \sim |y(s) - y_i(s,t(s))| = 
|\Psi_2 (s^n,\tilde y(s), t(s)) - \Psi_2 (s^n, y_i(s,0), t(s))| \sim |\tilde y(s) - y_i(s,0)|.
$$ 
This implies, again by the partial Lipschitz property of $\Psi_2$, that 
  $s^{\tilde l_i}  \sim (y(s,t) - y_i(s,t))$.  Therefore $y(u,t) - y_i(u,t)$, being analytic, equals 
 $u^{\tilde l_i}$ times a unit. 

Since $y(u,t) = O(u^n)$, $y_i(u,t) = O(u^n)$, and 
$(y(u,t) - y_i(u,t)) \sim u^{\tilde l_i}$,  the proof of \eqref{eq:distancewingpolar} can be obtained in a similar, even simpler, way as the formula \eqref{eq:ij1} of Proposition \ref{prop:distances}. 
\end{proof}

We set 
\begin{align*}
& l_i := \ord_s \dist (p(s), C_i) \leq  \tilde l_i:=\ord_s \dist (\pi(p(s)), \pi(C_i)), i\in I_C;\\
& l_k := \ord_s \dist (p(s), \Sigma_k) \leq  \tilde l_k:=\ord_s \dist (\pi(p(s)), \pi(\Sigma_k)), k\in I_\Sigma; \\
&  \text{and } \, l:=\max\{l_i, l_k\}, \,\, \tilde l :=\max\{\tilde l_i, \tilde l_k\}.
\end{align*}

\begin{prop}[Existence of quasi-wings I]\label{prop:existenceq-wing}
Assume that the arc $ p(s)$ satisfies
\begin{align}\label{eq:A2}
\forall i \in I_C, \,\,   m_i\geq \tilde l_i,
\end{align}
 and 
\begin{align}\label{eq:A3}
\forall j \in I:= I_C\cup I_\Sigma, \,\,    l_j=\tilde l_j.
\end{align}
Then, there is a regular wedge $W_q$ containing the projection $q(s)=\pi(p(s))$ and parameterized by 
$q(u,v,t)=( u^n, y(u,t) + vu^{\tilde l}, t), \,\, q(u,t):=q(u,0,t)$, satisfying $q(s,t(s))=q(s)$ and such that $\pi^{-1}(W_q)$ is a finite union of nicely situated quasi-wings. One of these quasi-wings contains $p(s)$.
\end{prop}

\begin{proof}
If we apply Lemma \ref {lem:wing} to  $q(s) = \pi (p (s))$  then we get $\tilde l_i= l_i,$ thus $l=\tilde l$ and therefore 
\begin{align*}
s^{l_i}  \sim \dist (\pi(p (s)), \pi (C_i)) \sim |y(s) - y_i(s,t(s))| 
\sim |\tilde y(s) - y_i(s,0)|. 
\end{align*}
A similar property holds for each component $\Sigma_k$ of the singular locus.  

The map
\begin{align*}
q(u, v,t) = ( u^n, y(u,t) + u^l v,t) ,
\end{align*}
for $v$ small, parameterizes a regular wedge $\horn_q$. 
The inverse image $\pi^{-1} (\horn_q)\cap \mathcal X$ is a finite union of nicely-situated quasi-wings and one of them contains $p(s)$. 
\end{proof}

\begin{cor}[Existence of quasi-wings II]\label{cor:existenceq-wing}
Suppose that $p(s) =(s^n, y(s), z(s), t(s))$ is  a real analytic arc 
in $\XX$ and not contained in the singular locus $\Sigma_f$. 
Then, for $b_0$ small and generic, $p(s)$ belongs to a quasi-wing in the coordinates 
$x, Y_{b_0},z,t$, where $Y_{b_0}:=y-b_0z$. 
\end{cor}

(Here by generic we mean in $\{b\in \C; |b|< \varepsilon \} \setminus A$, where $A$ is finite.  Moreover, we show that one may choose $\varepsilon>0$ independent of $p(s)$.)

\begin{proof}
Recall  that
$$\tilde l_i := \ord_s \dist (\pi (p (s)), \pi(C_i)), 
\qquad \tilde l_k := \ord_s \dist (\pi (p (s)), \pi (\Sigma_k)).$$
If all $\tilde l_i = l_i \le m_i, i\in I_C$,  $\tilde l_k = l_k, k\in I_\Sigma$ then the result follows from Proposition \ref{prop:existenceq-wing}.  Nevertheless, whether this is satisfied or not, it follows from Lemma \ref{lem:wing} that $\tilde l_i = \ord_s (y(s) - y_i(s,t(s)))$.  

We denote $\pi_b (x,y,z, t) := (x,y-bz,t)$ and by $C_{i,b}$ the associated polar set. By Transversality Assumption  $\mathcal X$ is Zariski equisingular with respect to $\pi_b$ for $b$ sufficiently small (that defines $\varepsilon$).  We claim that if $\tilde l_i > l_i$ 
\textcolor{black}{and $l_i \le m_i$} then the order $ \ord_s \dist (\pi_b (p (s)), \pi_b(C_{i}))= l_i$, for $b\ne 0$.  
Indeed, otherwise this order is strictly bigger than $l_i$  and then, again by 
Lemma \ref{lem:wing}, $|y(s) -y_i(s,t\textcolor{black}{(s)})  - 
b (z(s) -z_i(s,t\textcolor{black}{(s)}))| \ll s^ {l_i}$. By 
$\tilde l_i > l_i$ 
we have $|y(s) -y_i(s,t\textcolor{black}{(s)})|  \ll s^ {l_i}$ and therefore 
$ |z(s) -z_i(s,t \textcolor{black}{(s)})| \ll s^ {l_i}$  that 
contradicts $ \ord_s \dist (p (s), C_{i})= l_i$. 
Moreover, we claim that 
$ \ord_s \dist (\pi_b (p (s)), \pi_b(C_{i,\textcolor{black}b}))= l_i$, for $b\ne 0$ and small.  Indeed, by \eqref{eq:generalform-Yt},  
$$Y_b(s,b,t(s)) - (y(s)-bz(s)) = (y_i(s,t\textcolor{black}{(s)}) -y(s)) -b (z_i(s,t\textcolor{black}{(s)}) -z(s) ) + b^2 s^{m_i} 
 unit (s,b, t(s)). 
$$
The first summand is of size $s^{\tilde l_i}$, the second one of size 
$b s^{l_i}$, and the third one of size $b^2 s^{m_i}$. Therefore the claim follows for small $b\ne 0$ because $l_i \le m_i$.

If $l_i > m_i$ then 
$ \ord_s \dist (p (s), C_{i,b})= m_i$ for $b\ne 0$.  Therefore, in general, only for finitely many $b$, one for each $C_i$, we do not have $ \ord_s \dist (p (s), C_{i,b})\le m_i$.  

Finally, by a similar argument, $ \ord_s \dist (p (s), \Sigma_k) = 
\ord_s \dist (\pi_b (p(s)), \pi_b (\Sigma_k))$ for all $b$ but one. 

Thus the statement follows from Proposition \ref{prop:existenceq-wing}.  
\end{proof}

 \subsection{Basic properties of quasi-wings}

Let $p(s)$ be an arc as given in \eqref{A1} satisfying the assumptions of 
Proposition  \ref{prop:existenceq-wing} and let $\QW$ be the quasi-wing 
constructed in the proof of this proposition.  
 Let $p(u,v,t)=(q(u,v,t), z(u,v,t))$ be its parameterization.
Then, by  Lemma \ref{lem:wing}, $\tilde l_i{} =\ord_s(y(s)-y_i(s,t(s)))$
 and $\dist(p(s),C_i)\sim\dist (p(s), \PW_i) \sim s^{l_i}$ 
(and recall $\tilde l _i= l_i \ge m_i$).

 We shall show that the distances from $\QW$ to $\PW_i$ and to 
 $\Sigma_k$  are constant, that is, they are of order $u^{l_i}$ and  $u^{l_k}$ respectively.  
 This follows from their construction that uses arc-wise trivializations  
 of \cite{PP17} and the partial Lipschitz property of these trivializations, see Remark \ref{rmk:partialLipschitz}.  

 Recall that $\QW$ is constructed as follows.  
 Let \eqref{eq:arcwise-trivialization} be an arc-wise trivialization 
 preserving the discriminant locus $\Delta =0$. Then there is an arc 
  $q_0 (s)= (s^n, \tilde y(s),0)$ such that 
 $\Ph (u^n, \tilde y(u), t)$ is a complex analytic wing containing  
 $q(s)=\Ph (s^n, \tilde y(s), t(s))$.  The lift of $\Ph$ is an arc-wise analytic trivialization of $\XX$, see the proof of \cite[Theorem 3.3]{PP17}.  Let us denote this lift by 
 $$\Phh (x,y,z,t) = (\Psi_1 (x,t), \Psi_2 (x,y,t) , \Psi_3(x,y,z,t),t),$$
with  $\Psi_1(x,t)=x$.
 Let $p_0(s)$ denote the lift of $q_0(s)$.  Then  
$p(s)=p(s,t(s)) = \Phh (p_0(s),t (s))$.

The following proposition extends the conclusion of Lemma \ref{lem:wing} 
from the complex analytic wing $q(u,t)$ to the quasiwing $\QW$.

\begin{prop}\label{propB}
Let $\QW$ be the quasi-wing containing $p(s)$ given by Proposition  \ref{prop:existenceq-wing} and let $p(u,v,t)=(q(u,v,t), z(u,v,t))$ be its parameterization. Then for the polar sets $C_i$ parameterized by $p_i(u,t)$ and $\Sigma_k$ by $p_k(u,t)$,
$$(p(u,v,t)-p_i(u,t)) \sim u^{l_i}, \,\, (p(u,v,t)-p_k(u,t)) \sim u^{l_k}.$$
This implies that $\dist(p(u,v,t), \PW_i)\sim u^{l_i}$ and $\dist(p(u,v,t), \Sigma_k)\sim u^{l_k}.$
\end{prop}

\begin{proof}
It would be convenient in the proof to  use the constant $\varepsilon$ of Definition \ref{def:polarwedge} and denote for this constant fixed, i.e. for $|b|<\epsilon$, the polar wedges by $\PW_{i,\varepsilon}$ and 
by $\overline \PW_{i,\varepsilon}$ their closure.  
We denote by $\PW_\varepsilon$ (and by $\overline \PW_\varepsilon$) the union of $\PW_{i,\varepsilon}$  (respectively of $\overline \PW_{i,\varepsilon}$)  and the singular set $\Sigma_f$. 

\begin{lemma}\label{lem:preserv-polars}
$\tilde \Phi$ preserves the polar wedges in the following sense.
There is a constant $L$ (depending on the Lipschitz constant of $\Psi_2$ for its partial Lipschitz property, see Remark \ref{rmk:partialLipschitz}) such that
$$ \PW_{i,{\epsilon}/{L}}\subset \Phh (\PW_{i,\epsilon})
\subset  \PW_{i, L \epsilon} .$$
\end{lemma}\label{lem:preservingcontactsI}

\begin{proof}
By construction $\Phh$ preserves the polar set and the singular locus.  Therefore the lemma follows from the partial Lipschitz property of $\Psi_2$ and parameterization \eqref{eq:generalform-t}.
\end{proof}

\begin{lemma}
The following holds: 
$$\dist(\Phh(p_0(s),t), \PW_i)\sim s^{l_{i}}  \text {, } \quad 
\dist(\Phh(p_0(s),t), \Sigma_k)\sim s^{l_{k}} .$$    
\end{lemma}

\begin{proof}
Let $l = \max_{i\in I} l_i$.  First for fixed $\varepsilon >0 $ we show that 
\begin{align}\label {eq:distancetopolarset}
\dist(\Phh(p_0(s),t), \overline \PW_\varepsilon)\sim s^{l} .
\end{align}
It is clear that this distance is $\gtrsim$, this already holds after the projection $\pi$.  We show the opposite inequality. 

Fix $s_0>0$.  By Lemma \ref{lem:wing} 
$\dist (q_0(s_0), \pi (\overline \PW_\varepsilon) \cap \{t=0, s=s_0\} )
\sim s_0^l$.  
Let $c(s_0)$ be such that this distance equals exactly 
$c(s_0) s_0^l$ and let $q_{min}(s_0)$ be one of the points in 
$\pi (\overline \PW_\varepsilon) \cap \{t=0, s=s_0\}  $ realizing this distance.  Let $\tau$ be the lift of the segment joining $q_0(s_0)=\pi(p_0(s_0))$ and $q_{min}(s_0)$. 
Since $\tau $ is in the complement of $\PW_{\epsilon}$ (except if its endpoint is in $\Sigma_f$), by the boundness of partial derivatives, c.f. the argument of the proof of Proposition \ref{prop:deriv-bounded}, its length is comparable to the length of the segment, that is $s_0^l$.  
Denote by $p_{min}(s_0)$ the other endpoint of this lift, 
so that $q_{min}(s_0) =\pi(p_{min}(s_0))$.  
Since $\Psi_2$ is partially Lipschitz and $\Phh$ preserves the complement of $\PW_{\epsilon}$, 
see Lemma \ref{lem:preserv-polars}, we have for small $t$
\begin{align} \label{eq:dist-tau}
\dist(\Phh(p_0(s_0),t), \Phh(p_{min}(s_0),t))\lesssim s_0^l.
\end{align} 
Since the distance $c(s_0) s_0^ l$ is a subanalytic function we may suppose, 
by a choice of $q_{min}(s_0)$, that also $q_{min}(s_0)$ and $p_{min}(s_0)$ 
are subanalytic in $s_0$. 

There are three cases to consider 
$p_{min}(s_0)\in \overline \PW_\varepsilon\setminus \Sigma_f$, 
$p_{min}(s_0)\in \Sigma_f$, and $p_{min}(s_0)\notin \overline \PW_\varepsilon$.  

If $p_{min}(s_0)$ is in $\overline \PW_\varepsilon\setminus \Sigma_f$ then, since 
$\Phh$ preserves the polar set, so is $\Phh(p_{min}(s_0),t)$, and the claim follows from \eqref{eq:dist-tau}.  A similar argument applies if 
 $p_{min}(s_0)\in \Sigma_f$.  

If $p_{min}(s_0)\notin \overline \PW_\varepsilon$ then there is 
another point in $\pi^ {-1} (q_{min}(s_0))$ that is in 
$\overline \PW_\varepsilon$.  Suppose that it is in $\overline \PW_{j, \epsilon}$ and denote it by $p_j (s_0)$.  By the assumptions $l_j = \tilde l_j = \tilde l=l$ and by the partial Lipschitz property the magnitude of $\dist(\Phh(p_j(s_0),t), \Phh(p_{min}(s_0),t))$ 
is independent of $t$, say $\sim s_0^\alpha$.  
If  $\alpha\ge l$  then \eqref{eq:distancetopolarset} follows from 
\eqref{eq:dist-tau}. 
If  $\alpha< l$ then 
$\dist(\Phh(p_j(s_0),t), \Phh(p_{min}(s_0),t)) \sim \dist (
\overline \PW_{j} ,\Phh(p_{min}(s_0),t))$ and 
therefore $\dist(\Phh(p_j(s_0),t), \Phh(p_{0}(s_0),t)) \sim \dist (
\overline \PW_{j} ,\Phh(p_{0}(s_0),t))$.  But, by assumption on the curve $p(s) = \Phh (p_0(s), t(s))$, 
\begin{align*}
 & \dist(\Phh(p_j(s_0),t(s_0)), \Phh(p_{min}(s_0),t(s_0))) \\
 & \le 
\dist(\Phh(p_j(s_0),t(s_0)), p(s_0)) + \dist(p(s_0), \Phh(p_{min}(s_0),t(s_0)))  
\le C s_0^l ,
\end{align*}
for a universal constant $C$.  This shows that the case $\alpha< l$ 
is impossible.

Now we show that \eqref{eq:distancetopolarset} implies the claim of  lemma.  Again, it is enough to show $\lesssim$ since the opposite inequality is already known for the sets projected by $\pi$.  Firstly, the distance on the left-hand side of \eqref{eq:distancetopolarset} has to be attained on one of $\overline \PW_{j,\varepsilon}$ or $\Sigma _k$.  Suppose, for simplicity, that it is 
$\overline \PW_{j,\varepsilon}$.  Then $l=l_j$, that implies the claim of lemma for $i=j$.  By the above there is a curve $p_j(s) \in 
\overline \PW_j \cap \{t=0\}$ such that 
\begin{align} \label{eq:dist-tau-tau}
\dist(\Phh(p_0(s),t), \Phh(p_{j}(s),t))\sim s^{l_ j}.
\end{align}

Let $i\ne j$.  Then  $l_i\leq l_j$ and 
\begin{align} \label{eq:dist-tau-tau-tau}
\dist(\Phh(p_0(s),t), \overline \PW_i)\lesssim s^{l_j} + \dist(\Phh(p_j(s),t), \overline \PW_i).
\end{align}
To complete the proof we note that $\dist(\Phh(p_j(s),t), \overline \PW_i) 
\sim s^ {k_{ij}}$ and $k_{ij}$ is also the order of contact between the  discriminant branches $\Delta_i$ and $\Delta_j$.  
If $l_i<l_j$ then $\dist(q(s,t),\Delta_i)\sim\dist(\Delta_i,\Delta_j)
\sim s^{k_{i,j}},$ and by \eqref{eq:A3}, $l_i=\tilde l_i=k_{i,j}.$

If $l_i=l_j$ then $k_{i,j}<l_i=l_j$ is impossible. 
Thus $k_{i,j}\geq l_j$ and the RHS of \eqref{eq:dist-tau-tau-tau} is bounded by $s^{l_i}=s^{l_j} $ as claimed.  
This ends the proof of Lemma 7.11.  
\end{proof}

To show Proposition \ref{propB} we note that  $(y_i(u,t)-y(u,t))\sim u^{l_i}$ 
by Lemma \ref{lem:wing} and $z_i(s,t)-z(s,t)$ is divisible by $s^{l_i}$ for $s$ real and hence $z_i(u,t)-z(u,t)$ is divisible by $u^{l_i}$.  
\end{proof}  

\begin{cor}\label{cor:orders-l_i}
Under the assumption of Proposition \ref{propB}, we have 
$$(y_i(u,t)-y(u,t))\sim u^{l_i} \text{ and } \,\, z_i(u,t)-z(u,t)=O(u^{l_i})$$ 
for all $i\in I=I_C\cup I_\Sigma$.    \qed
\end{cor}

\noindent


\section{Lipschitz vector fields on quasi-wings}

Let the quasi-wings $\QW_\tau$ over a fixed regular wedge $\horn_q$ parameterized by 
\eqref{eq:param-q} be given by 
\begin{align}\label{eq:q-wing-parameterization}
p_\tau (u,v,t) = ( u^n, y (u,v,t ), z_\tau(u,v,t),t), \,  \, \textcolor{black}{y(u,v,t)=y(u,t) + u^l v}.
\end{align}  
We consider such parameterizations for $u$ in an allowable sector $\Xi = \Xi_I= \{u\in \C; \arg u \in I\}$.  Then we may write these parameterizations in terms of $t,x,v$ assuming implicitly that we work over a sector $\Xi$ and, moreover, 
 that $z_\tau (x,v,t )$ is a single valued functions.  Again, in order to avoid heavy notation we do not use special symbols for the restriction of a quasi-wing parameterization to an allowable sector.

 Even if the parameterizations of quasi-wings carry many similarities to the parameterizations of polar wedges,  the boundness of partial derivatives (the property (4) of the beginning of the previous section)  is opposite to the very definition of polar set, the vertical tangent versus the horizontal tangents. This boundness and the fact that the projection $\pi$ restricted to a quasi-wing is bi-Lipschitz make the work with the Lipschitz geometry of quasi-wings in principle simpler.

\begin{prop}\label{prop:distances-qw}
For all $\tau$   and for all $x_1,x_2,v_1,v_2, t_1,t_2$ sufficiently small 
\begin{align}\label{eq:i=j-qw}
\|p_\tau (x_1,v_1,t_1) -p_\tau (x_2,v_2,t_2) \| & \sim 
\textcolor{black}{\|(x_1,y_1,t_1) - (x_2,y_2,t_2)\|} \\  \notag
& \sim \max \{|t_1-t_2|, |x_1-x_2|,  |v_1-v_2|  |x_2 |^{l/n}  \} .
\end{align}
For every pair of parameterizations $p_\tau , p_\nu$ 
\begin{align}\label{eq:ij-qw}
& \|p_\tau (x_1,v_1,t_1) -p_\nu (x_2,v_2,t_2) \| \\ \notag 
& \sim  \|p_\tau (x_1,v_1,t_1) -p_\tau (x_2,v_2,t_2) \| + \|p_\tau (x_2,v_2,t_2) -p_\nu (x_2,v_2,t_2) \| \\
\notag
& \sim \max \{|t_1-t_2|, |x_1-x_2|,  |x_2 |^{r_{\tau \nu}/n}, |v_1-v_2|  |x_2 |^{l/n}   \} ,
\end{align}
where $r_{\tau \nu}$ are given by \eqref{eq:generalformpairs2}.  \qed
\end{prop}

By Proposition \ref{prop:distances-qw}, $h_\tau (x,v;t)$ defines a Lipschitz function on the quasiwing 
$\QW_\tau$ if and only if 
 \begin{align}\label{eq:qw-lipschitz}
|h_\tau(x_1,v_1,t_1) -h_\tau(x_2,v_2, t_2) | & \lesssim  \|(x_1,y_1,t_1) - (x_2,y_2,t_2)\|
 \\  \notag 
& \sim  
 |t_1-t_2| + |x_1-x_2| +  |v_1-v_2|  |x_2 |^{l/n}  .
\end{align}

Given two nicely-situated quasi-wings.  
Let $h$ be a function defined on a subset of $\QW_\tau \cup \QW_\nu$ whose restrictions to 
$\QW_\tau$, $ \QW_\nu$ we denote by   \textcolor{black}{$h_\tau (x,v,t)=h\circ p_\tau$, $h_\nu (x,v,t)=h\circ p_\nu$ respectively}. 
 Then, after Proposition \ref{prop:distances-qw}, $h$ is Lipschitz iff so are its restrictions $h_\tau$, $h_\nu $ and
\begin{align}\label{eq:2qw-lipschitz}
|h_\tau(x_1,v_1, t_1) -h_\nu(x_2,v_2; t_2) | \lesssim 
|t_1-t_2| + |x_1-x_2| + |x_2|^{r_{ij}/n}+ |v_1-v_2|  |x_2 |^{l/n}   .
\end{align}

\begin{prop} \label{prop:LVF1-qw}
The vector fields given on $\QW_\tau \cup \QW_\nu$ by  $p_{k*}(v)$, $k=\tau, \nu$, 
where $v$ are $\frac {\partial}{\partial t}$, $x\frac {\partial}{\partial x}$, or $\frac {\partial}{\partial v}$,  are Lipschitz.  
\end{prop} 

This result is  analogous to Proposition \ref{prop:LVF1}.  The only  difference comes from the fact that $b\frac {\partial}{\partial b}$ is replaced by 
$\frac {\partial}{\partial v}$, since we do not require the vector field to be tangent 
to the set given by $v=0$.  The proof we sketch below is simpler that the one of 
Proposition \ref{prop:LVF1} thanks to the mentioned above bi-Lipschitz property.  

\begin{proof}
\textcolor{black}{First we check that the partial derivatives $\frac {\partial}{\partial t}$, 
$x\frac {\partial}{\partial x}$, 
$\frac {\partial}{\partial y}$ 
of the coefficients of these vector fields  are bounded. 
Since $nx\frac {\partial}{\partial x} =  u\frac {\partial}{\partial u}$ and $\frac {\partial}{\partial y} = u^{-l}\frac {\partial}{\partial v}$ for the latter two it is more convenient} to check that 
$u\frac {\partial}{\partial u}$ is bounded by $x=u^n$, and 
$\frac {\partial}{\partial v}$ is bounded by $u^l$.  
Then the claim follows from the facts that $y (u,v,t ), z_\tau(u,v, t)$ 
are analytic and divisible by $u^n$, and $\frac {\partial}{\partial v} y (u,v,t ), 
\frac {\partial}{\partial v} z_\tau(u,v, t)$ are divisible by $u^l$. 
This shows that these vector fields are Lipschitz on each wing $\QW_\tau$, 
$\QW_\nu$.

To obtain the Lipschitz property between the points of $\QW_\tau$ and  
$\QW_\nu$ we use a similar argument.  \textcolor{black}{ Namely, we use formula \eqref{eq:generalformpairs2} to show that} 
$\frac {\partial}{\partial t}(z_\tau -z_\nu)$, 
$\frac {\partial}{\partial u} (z_\tau -z_\nu)$, 
$\frac {\partial}{\partial v} (z_\tau -z_\nu)$ are bounded (up to a constant) by 
$z_\tau -z_\nu$, and we complete using formulas \eqref{eq:i=j-qw} and \eqref{eq:ij-qw}.
\end{proof}

 Let   $p_{\tau, *} (w)$ be a vector field on $\QW_ \tau$, where 
\begin{align}\label{eq:vectorfieldw}
w(x,v,t)= \alpha \frac {\partial}{\partial t} + \beta \frac {\partial}{\partial x} + 
\gamma \frac {\partial}{\partial v}.
\end{align}
\textcolor{black} {We always suppose the vector field $p_{\tau, *} (w)$ is 
well defined on $\QW_ \tau$, that is independent of $v$ if $x=0$,
 and  it is stratified,  that is tangent to $T$.    
  The independence on $v$ if $x=0$ implies that both $\alpha (0,v,t)$ and $\beta (0,v,t)$ are independent on $v$, and the tangency to $T$ assures that  in fact $\beta (0,v, t)=0$. Note also that 
 $ p_{i*}(\frac {\partial}{\partial v})$ is always zero on $x=0$.}

The next results easily follow from \eqref{eq:qw-lipschitz}.  Their proofs are similar (and simpler) then the proofs of Propositions \ref{prop:LVFcriterion1} and \ref{prop:LVFcriterion2}.

\begin{prop}\label{prop:LVFcriterion1-qw}
A vector field on  $\QW_\tau$ of the form $p_*(w)$ 
 is stratified  Lipschitz iff: \\
1) $ \alpha $ satisfies \eqref{eq:qw-lipschitz};\\
2) $|\beta |\lesssim |x|$ and 
$ \beta $ satisfies \eqref{eq:qw-lipschitz};\\
3) $ \gamma x^{l/n}$  satisfies \eqref{eq:qw-lipschitz}.  
\qed 
\end{prop}

\begin{prop}\label{prop:LVFcriterion2-qw}
A vector field  on $\QW_\tau \cup \QW_\nu$ given by 
$p_{\tau*}(w_\tau)$, $p_{\nu*}(w_\nu)$ is stratified  Lipschitz  iff: \\
 0) $p_{\tau*}(w_\tau)$ and $p_{\nu*}(w_\nu)$ are  Lipschitz;\\
1) $ \alpha _\tau, \alpha_\nu $ satisfy \eqref{eq:2qw-lipschitz};\\
2) $ \beta_\tau, \beta_ \nu$  satisfy \eqref{eq:2qw-lipschitz};\\
3) $ \gamma_\tau x^{l/n},  \gamma_\nu x^{l/n}$  satisfy \eqref{eq:2qw-lipschitz}.  
\qed 
\end{prop}

We now consider the extension of Lipschitz vector fields on quasi-wings.  
The classical McShane-Whitney extension theorem, \cite[Theorem 1]{mcshane1934}, \cite[the footnote on p. 63]{whitney34}, says that 
a Lipschitz function $\varphi: A\to \R $ defined on any nonempty subset $A$ 
of a metric space $B$ admits a Lipschitz extension $\tilde \varphi$ to $B$ with the same 
Lipschitz constant.  (Such an extension can be even given by a formula $\tilde \varphi (x)= \inf_{y\in A} (\varphi (x) + Lip(\varphi) d(x,y))$.) If $B\subset \R^n,$ then this theorem gives an extension of Lipschitz vector fields with the Lipschitz constant multiplied by $\sqrt n$.  The Kirszbraun Theorem, see e.g.   
\cite[p. 202]{federerbook}, shows the existence of an extension of vector fields 
with the same Lipschitz constant. In our case we can use any of these results.  By Proposition \ref{prop:LVFcriterion1-qw},  
$w\to p_* (w) $ gives a one-to-one correspondence between  Lipschitz vector fields on $\QW$ and Lipschitz vector fields 
$w(x,y,t)$ on the wedge $\W$.  Hence  the McShane-Whitney extension theorem implies the following.

\begin{cor}[Extension of Lipschitz vector fields on a quasi-wing]\label{cor:extension-QW1}
Any stratified Lipschitz  vector field defined on subset of a quasi-wing $\QW$ containing the stratum $T=\{x=0\}$ can be extended to a stratified Lipschitz  vector field on $\QW$.
\qed  
\end{cor}

Propositions \ref{prop:LVFcriterion1-qw}, \ref{prop:LVFcriterion2-qw} imply the following.

\begin{cor}[Extension of Lipschitz vector fields between the quasi-wings]\label{cor:extension-QW2}
Let $\QW_\tau, \QW_\nu$ be nicely-situated quasi-wings parameterized by 
$p_\tau (x,v,t)$ and $p_\nu (x,v,t)$ respectively.  
Let the vector field $w $, of the form \eqref{eq:vectorfieldw}, be  such that 
$p_{\tau*}(w)$ is a stratified Lipschitz vector field defined on the image of $p_{\tau}$. 
Then  $p_{\tau*}(w)$, $p_{\nu*}(w)$, define a stratified Lipschitz vector field on the union  
$\QW_\tau \cup \QW_\nu$.  
\qed  
\end{cor}


\section{Extension of Lipschitz vector fields from $\PW$ to an arc in its complement}

Suppose we are given a stratified Lipschitz vector field $w$ on $S$.  
By the first part of the proof of Theorem \ref{thm:maintheorem}, Section 
\ref{sec:proof-partI}, we may extend it to a Lipschitz vector field, still called $w$, onto $\PW$.  In this section we show how to extend it further on 
the image of a real analytic arc germ $p(s)$  of the form 
\eqref{A1} not included in $\PW$.  For this we use Corollary \ref{cor:existenceq-wing} to embedd $p(s)$ in a quasi-wing $\QW$ and extend 
the vector field from $\PW$ to $\QW$.  The latter extension is explained 
in Proposition \ref{prop:ext:pw-->qw}. 
In the process we encounter two problems, discussed below,  
related to the fact that the construction of Corollary \ref{cor:existenceq-wing} gives a 
quasi-wing after a linear change of coordinates.

If $\PW_i$ is a polar wedge in the original system of coordinates then we may choose the corresponding polar wedge in the new system 
of coordinates $x, y-b_0z,z,t$,  denoted  by  $\PW_{i,b_0}$, 
  included in $\PW_i$, but we cannot assume that it contains the spine of $\PW_i$, 
that is $C_i$.  Therefore, if we extend $w|\PW_{i,b_0} $ to $\QW$ using 
Proposition \ref{prop:ext:pw-->qw}, a priori there is no guarantee that the obtained 
vector field is Lipschitz on $\PW_i \cup \QW$.  To guarantee it we show that 
the distance from the arc $p(s)$, and hence from the whole quasi-wing $\QW$, 
to $\PW_i$ and to $\PW_{i,b_0} $ are of the same orders.  This will follow from 
Proposition \ref{prop:distancearctoCj}. 

The second problem comes from the fact that the description 
of stratified Lipschitz vector fields on a polar wedge, given in the conditions 
1)-3) of Proposition \ref{prop:LVFcriterion1},  change slightly when 
we pass from $\PW_i$ to $\PW_{i,b_0}$,  if $\PW_{i,b_0}$ does not contain $C_i$.   
Therefore to show  Proposition \ref{prop:ext:pw-->qw} one should not use the condition 3). To solve this problem we replace in the proof of Proposition \ref{prop:ext:pw-->qw} the condition 3) by a slightly weaker condition 3') that is satisfied on $\PW_{i,b_0}$.

 \subsection{Distance to polar wedges}

\begin{prop}\label{prop:distancearctoCj}
Let $\gamma(s) =(x(s), y(s), z(s), t(s))$, $s\in [0,\varepsilon)$, be a real analytic arc at the origin. 
If $\gamma (s) \not \subset \PW $ then for all $j$, 
$$\dist (\gamma (s), C_j) \gtrsim \|(x(s), y(s), z(s))\|^{m_j/n}.$$
\end{prop}

\begin{rmk}
If the arc $\gamma$ is of the form $\gamma(s) =(s^n, y(s), z(s))$ with $y(s) = O(s^n), z(s)=O(s^n)$, that we may suppose, then we get that $\dist (\gamma (s), C_j) \gtrsim |s ^{m_j}|$.
\end{rmk}

For the proof of Proposition \ref{prop:distancearctoCj} we need the following lemma.  

\begin{lemma}\label{lem:distancearctoS}
If the polar set $C_i$ minimizes the distance of $\gamma$ to $S$ and if this distance satisfies 
\begin{align}\label{eq:condition}
\dist (\gamma(s) , S)=  \dist (\gamma(s) , C_i) \ll  
\|(x(s), y(s), z(s))\|^{m_i/n},  \end{align} 
then $\gamma (s)$ is contained, for small $s$, in $\PW$.  
\end{lemma}

By \eqref{eq:condition} we mean that there is  $\delta>0$ such that 
$$\dist (\gamma(s) , C_i ) \le \|(x(s), y(s), z(s))\|^{\delta + m_i/n} .$$

We do not claim in the lemma that $\gamma (s)$ has to belong to the polar wedge containing $C_i$, that is $\PW_i$.

\begin{proof}
We write the proof in the non-parameterized case. The proof in the parameterized case is similar.  

We may suppose that the arc $\gamma$ is of the form 
$\gamma(s) =(s^n, y(s), z(s))$ with $y(s) = O(s^n), z(s)=O(s^n)$ and note that in this case 
$\dist (\gamma(s) , C_i ) \sim |y(s)-y_i(s)| + |z(s)-z_i (s)| $.  Therefore, 
by \eqref{eq:condition}, 
$|y(s)-y_i(s)| = o (s^{m_i})$ and $|z(s)- z_i(s)| = o (s^{m_i})$.  
  Complexify $\gamma$ by setting $\gamma(u) =(u^n, y(u), z(u))$.  Then, 
as in the proof of Corollary \ref{cor:existenceq-wing}, we construct a quasi-wing $\QW$ containing $\gamma$ by changing the system of coordinates, that is replacing $y$ by $Y= y - b_0 z$, for $b_0$ sufficiently generic. 
 In this new system of coordinates $x,Y,z,t$ (we do not change the parameter $b$) the parameterizations of $\PW_i$ and $\QW$ are,  $x=u^n$ and, respectively,    
\begin{align}\label{eq:generalform-pw}
& Y_i (u,b)  = y_i(u,b) -b_0 z _i(u,b) \\
\notag 
& \qquad \quad  = (y_i(u) -b_0z_i (u)) + u^{m_i}
(b^2 \varphi_i (u,b) -bb_0 \psi_i (u,b)), \\
\notag
& z _i(u,b)  = z_i(u) + b u^{m_i} \psi_i (u,b). 
\end{align}
\begin{align}\label{eq:generalform-qw}
& Y(u,v) = (y(u) -b_0z (u)) + v u^{m_i}, \\
\notag
& z(u,v) = z(u) + v u^{m_i} \tilde \psi_i (u,v). 
\end{align}
To see that the exponent in the latter formula is $m_i$ note that, if we denote the polar set in $\PW_i$ in the new system of coordinates by $C_{i,b_0}$ then 
$\dist (\gamma(s) , C_{i,b_0}) \sim s^{m_i}$ and we conclude by  
Corollary \ref{cor:orders-l_i}.  
Now we argue as follows. By Proposition \ref{prop:deriv-bounded} the polar wedge $\PW_i$ and the quasi-wing $\QW$ are disjoint (if the constants defining them are small). But if the limit of tangent spaces to $\mathcal X$ along $C_i$ and along $\gamma$ do not coincide then the implicit function theorem forces $\PW_i$ and $\QW$ to intersect along a  curve and therefore this case cannot happen.  This is the geometric idea behind the computation below.

Note that \eqref{eq:condition} implies that, for the old system of coordinates, $l_i >m_i$.  Therefore the intersection $\PW_i\cap \QW$, defined by $Y_i(u,b)=Y(u,v)$ and 
$z_i(u,b)=z(u,v)$,  is given by the following system of equations   
\begin{align}\label{eq:system}
& (b^2 \varphi_i (u,b) -bb_0 \psi_i (u,b)) - v = {O}(u), \\
\notag
& b \psi_i (u,b) - v \tilde \psi_i (u,v)= {O}(u). 
\end{align}
There are two cases: 
\begin{enumerate}
\item [(i)]
\textcolor{black}{Suppose the jacobian determinant of the LHS of \eqref{eq:system}, with respect to variables $b,v$ is nonzero at $u=b=v=0$.
Then,} by the Implicit Function Theorem there is a solution $(b,v) = (b(u),v(u)) $ of 
\eqref {eq:system}, such that $b(u) \to 0$ and $v(u)\to 0$ as $u\to 0$.  
Then the intersection $\PW_i\cap \QW$ \textcolor{black}{is the curve parameterized by $u$:} $(u^n, Y_i(u,b(u)), z_i(u,b(u))) = (u^n,Y(u,v(u)),z(u,v(u)))$.
Therefore, by Proposition \ref{prop:deriv-bounded},   
 this case cannot happen.
\item [(ii)]
Suppose that the jacobian determinant of the LHS of \eqref{eq:system} vanishes at $u=b=v=0$.  Then the partial derivatives 
$$
\frac \partial {\partial b} \, u^{-m_i}(Y_i(u,b), z_i(u,b)), \quad 
\frac \partial {\partial v} \, u^{-m_i}(Y(u,v), z(u,v)),
$$
that are both non-zero at $u=b=v=0$, are proportional. This means that the limits of tangent spaces to $\mathcal X$ along $C_i$, i. e. at $(u^n, y_i(u,0), z_i(u,0))$ as $u\to 0$,  and at $\gamma (u)$ as $u\to 0$, coincide.  
This limit  is transverse to $H=\{x=0\}$ since 
$H$ is not a limit of tangent spaces by the Transversality Assumptions.  Hence the tangent spaces to $\mathcal X$ at $\gamma(u)$, for small $u$, 
contain vectors of the form $(0,b,1)$ with $b\to 0$ as $u\to 0$.  
This shows that $\gamma \in \PW$ (but not necessarily $\gamma \in \PW_i$). 
\end{enumerate} 
The proof of lemma is now complete.  
\end{proof}

\begin{proof}[Proof of Proposition \ref{prop:distancearctoCj}]
The proof is the same in the parameterized and the non-parameterized case.  
We may suppose again that $\gamma(s) =(s^n, y(s), z(s))$ with 
$y(s) = O(s^n), z(s)=O(s^n)$. 

If $\dist (\gamma(s) , S )=  \dist (\gamma(s) , C_i)$ then the conclusion for 
 $j=i$ follows directly from Lemma \ref{lem:distancearctoS}.  
Then consider $j\ne i$.  If the conclusion is not satisfied then 
$$
s^ {m_i} \lesssim \dist (C_i, \gamma (s)) \le \dist (C_j, \gamma (s)) \ll s^{m_j}.
$$
In particular, $m_i> m_j$, and therefore by Remark \ref{rmk:rem-mi-mj-kij}, 
$k_{ij}\le m_j<m_i$.  But this is impossible since then 
$$
s^ {m_j} \lesssim s^{k_{ij}} \simeq \dist (p_i(s),p_j(s)) \lesssim
\dist (C_j, \gamma (s)) + \dist (C_i, \gamma (s)) \ll s^{m_j},
$$
where $p_i, p_j$ denote parameterizations of $C_i$ and $C_j$ respectively.  
  This ends the proof in this case.  

If $\dist (\gamma(s) , S )=  \dist (\gamma(s) , \Sigma_k)$ 
then the conclusion follows by the second part of Lemma \ref{lem:distance_to_Sigma}.  
\end{proof}


\subsection{Extension of Lipschitz vector fields from a polar wedge to a quasi-wing}

Let the quasi-wing $\QW$ be given by 
$$\QW: p(u,v,t)= (u^n,y(u,t)+vu^l, z(u,v,t),t), \quad y(u,v,t):=y(u,t)+vu^l ,$$
containing an arc $p(u,t) = p(u,0,t)$.  

Fix a polar wedge $\PW_i$ (or $\Sigma_k$) closest to $\QW$ and parameterized by 
$$\PW_i : p_i(u,b,t)= (u^n,y_i(u,b,t),z_i(u,b,t),t).$$ 
Recall after Definition \ref{def:qwings} that $m_i\geq l=l_i$ and then by Corollary \ref{cor:orders-l_i}
\begin{align}\label{eq:recall}
 (y_i(u,b,t)-y(u,v,t))\sim u^l, \,\, \text{and}\,\, \, z_i(u,b,t)-z(u,v,t)=O(u^l).
 \end{align} 

Our goal is to extend any Lipschitz  stratified vector field 
 on $\PW_i$ onto $\QW$. Recall, after Proposition \ref{prop:LVFcriterion1}, that if $p_{i*}( \alpha \frac {\partial}{\partial t} + \beta \frac {\partial}{\partial x} + 
\delta \frac {\partial}{\partial b})$ is Lipschitz stratified then $\alpha$, 
$\beta$, and $\delta$ satisfy the conditions 1)-3) of Proposition \ref{prop:LVFcriterion1}.  In what follows we use only a weaker version of condition 3) that is,  see Remark \ref{rem:delta-smaller-than-b} for explanation,  

\medskip
\noindent
\emph{3')  $|\delta |$ is bounded and 
$ \delta x^{m/n}$  satisfies \eqref{eq:wedgelipschitz}.} 

\medskip We note that by \eqref{eq:recall} and $m_i \ge l$, a vector field 
 is Lipschitz on  $\PW_i \cup \QW$ if and only if it is Lipschitz on each 
 $\PW_i$ and $\QW$ and it is Lipschitz on the union of the images of two arcs 
 $p(u,t)$ and $p_i(u,t)$.  

\begin{prop}[Extension of Lipschitz vector fields from $\PW_i$ onto $\QW$.]\label{prop:ext:pw-->qw}
Let $p_{i *}( \alpha (u,b,t)  \frac {\partial}{\partial t} 
 + \beta (u,b,t)  \frac {\partial}{\partial x} + 
\delta (u,b,t)  \frac {\partial}{\partial b})$ be a stratified Lipschitz vector field on $\PW_i$. Set 
$\alpha_0 (u,v,t) := \alpha(u,0,t)$ and $\beta_0 (u,v,t) := \beta(u,0,t)$. 
Then $p_{*}( \alpha_0 \frac {\partial}{\partial t} + \beta_0 \frac {\partial}{\partial x})$ is a stratified Lipschitz vector field on $\QW$ and both fields define a stratified Lipschitz vector field on $\PW_i\cup\QW.$
\end{prop}

\begin{proof}

 By Proposition \ref{prop:LVFcriterion1-qw},  $p_{*}( \alpha_0 \frac {\partial}{\partial t} + \beta_0 \frac {\partial}{\partial x})$  is Lipschitz on $\QW.$ To show that both vector fields define a Lipschitz vector field  on $\PW_i\cup\QW$ it suffices to show that, taking $b=0$ and $ v=0$ we have:
\begin{enumerate}
\item $\alpha (u,0,t)\frac {\partial}{\partial t}(y(u,t)-y_i(u,t))=O(u^{l_i});$
\item $\alpha (u,0,t)\frac {\partial}{\partial t}(z(u,t)-z_i(u,t))=O(u^{l_i});$
\item $\beta  (u,0,t)\frac {\partial}{\partial u}(y(u,t)-y_i(u,t))=O(u^{l_i});$
\item $\beta  (u,0,t)\frac {\partial}{\partial u}(z(u,t)-z_i(u,t))=O(u^{l_i});$
\item $\delta (u,0,t) u^{m_i}=O(u^{l_i}).$
\end{enumerate}
The  items (1)-(4) follow from \eqref{eq:recall} and (5) follows from $m_i\geq l_i.$
\end{proof}

\begin{rmk}\label{rem:delta-smaller-than-b}
Since in the above proof we only used the condition 3') we can  apply Proposition \ref{prop:ext:pw-->qw} to the quasi-wings constructed in Corollary \ref{cor:existenceq-wing}, that is  after a change of coordinates to $x, Y_{b_0},z,t$, 
where $Y_{b_0}:=y-b_0z$, that corresponds to a shift in $b$.    
\end{rmk}


\section{Proof of Theorem \ref{thm:maintheorem}.  Part II}

We complete the proof of Theorem \ref{thm:maintheorem}.  Let $\gamma (s), \gamma'(s)$, $s\in [0,\varepsilon)$, be 
 two real analytic arcs in $\mathcal X$.
    We want to show that any stratified Lipschitz vector field $v$ defined on the union of $S$ and $\gamma$ extends to 
 $\gamma'$ as stated in the valuative criterion, see the next section.  We consider two cases.  

\medskip
\noindent
\textbf{Case 1.} 
$\dist (\gamma (s), \gamma '(s)) \gtrsim \dist (\gamma' (s), S)$.  \\
\textcolor{black}{Then it is enough to extend $v|_S$ to a Lipschitz vector field on $S\cup \gamma'$, since then such an extension defines a Lipschitz vector field on $S\cup \gamma(s) \cup \gamma'(s)$  for every $s$ sufficiently small, with the Lipschitz constant independent of $s$}. \\

\noindent
\textbf{Case 2.} 
$\dist (\gamma (s), \gamma '(s)) \ll \dist (\gamma' (s), S)$.  
\textcolor{black}{Then it suffices to extend $v$ from $\gamma$ to a Lipschitz vector field on $\gamma \cup \gamma'$.} \\
Note that we may suppose that on  both arcs $\gamma$, $\gamma'$ we have that $y= O(x), z= O (x)$, 
 that is, they are in the form \eqref{eq:gammaform}. Indeed, by Transversality Assumption the variable $z$  restricted to an arc in  $\mathcal X$ cannot dominate $x$ and $y$, that is $x= o(z), y= o (z)$ is not possible.  Thus, if $y= O(x), z= O (x)$ is not satisfied, then  $x= o(y), z= O (y)$.  In this case  we change the local coordinate system to $(X_a,y,z,t) = (x-ay,y,z,t)$, for $a\ne 0$ and small.  
 This is a change of coordinates in the target of the projection $(x,y,z,t) \to (x,y,t)$ and  affects neither the discriminant locus  nor Zariski's  Equisingularity.

To make the proof more precise we will use the constant $\varepsilon$ of Definition \ref{def:polarwedge} and denote thus defined the union of polar wedges and the singular set by 
$\PW_\varepsilon$.  If both $\gamma (s), \gamma'(s)$ belong to $\PW_\varepsilon$ then the claim follows from the first part of the proof, Section \ref{sec:proof-partI}. 

\textcolor{black}{
In  \textbf{Case 1}, given a stratified Lipschitz vector field $v$ on $S$ we extend it on $\gamma'$.  By Proposition \ref {prop:distancearctoCj} we may suppose  $\dist (\gamma (s), C_j) \gtrsim   s^ {m_j}$ for every $j$, and therefore, for $b$ small, say 
$b\le \varepsilon$, $\dist (\gamma (s), C_j)  \sim  \dist (\gamma (s), C_{j,b})$, where $C_{j,b}$ denotes the polar set in $\PW_j$ after  the  change of coordinates to $x, Y_{b_0}= y-b_0 z,z,t$. Then we proceed as follows.  First we extend $v$ to a Lipschitz vector field on  $\PW_{\varepsilon/2 }$ and use Corollary \ref{cor:existenceq-wing} to embedd $\gamma '$ in a quasi-wing in this new system of coordinates for a $b_0\le \varepsilon /2$.   Thus there exists a quasi-wing $\QW$ containing $\gamma ' $ and, moreover, $\dist (\gamma'(s) , S) = \dist (\pi_{b_0} 
(\gamma'(s)) , \Delta_{b_0}) \sim s^ l$, where $l= \max \{ \max l_i, \max r_k\}$ and $\Delta_{b_0}$ denotes the discriminant $\pi_{b_0}$. 
 Then there is a Lipschitz extension of $v$ to $\QW$ by Proposition \ref{prop:ext:pw-->qw}.  }

Similarly, in  \textbf{Case 2} we may suppose  
$\dist (\gamma (s), C_j) \sim \dist (\gamma ' (s), C_j) \gtrsim   s^ {m_j}$ for every $j$, 
otherwise, by Proposition \ref {prop:distancearctoCj}, both $\gamma (s), \gamma'(s)$ belong to $\PW_\varepsilon$.  Then, conveniently choosing $b$, we may suppose that 
$\dist (\pi_b(\gamma(s)),\pi_b(\gamma'(s))) \sim \dist (\gamma(s),\gamma'(s)) \ll s^l$.  Let $\QW$ be a quasi-wing containing $\gamma$. It always exists by  Corollary \ref{cor:existenceq-wing}, and $\gamma'$ is contained either in $\QW$ or in another quasi-wing 
$\QW'$ such that $\QW$ and $\QW'$ are nicely-situated.  Then we apply 
Corollary \ref{cor:extension-QW2} to extend a Lipschitz vector field $v $ from $\gamma$ to $\gamma'$.

\section{Valuative criterion on extension of Lipschitz vector fields}
\label{sec:LVEC}
The purpose of this section is to give a precise statement 
of a valuative criterion on extension of Lipschitz vector fields.  In this 
criterion we formalise our strategy of checking the conditions (i) and (ii) of Proposition \ref{prop:LipStr-charcterization} along real analytic arcs.  

Let us consider the following more general set-up.  Let $X$ be a locally closed subanalytic subset $\R^n$ with a filtration $\SX
=(X^j)_{ j=l,\dots , d}$ by closed subanalytic subsets 
\begin{align}\label{eq:filtration}  
X=X^d \supset X^{d-1} \supset \cdots \supset X^l \ne\emptyset ,
\end{align}
such that for every $j=l,\dots , d$, $ \mathring{X}^{j}=X^{j}\setminus X^{j-1}$ is either empty or a real analytic submanifold of pure dimension $j$. 
Here we mean $X^{l-1} =\emptyset$. Note that $\SX$ induces a stratification of $X$ by taking the connected components of every $\mathring{X}^{j}$ as strata.    By a stratified Lipschitz vector field (SLVF for short) 
we mean a Lipschitz vector field defined on a subset of $X$ and tangent to the strata.

\begin{defi}\label{def:LVE}
[Local Valuative Extension of Lipschitz Vector Fields Condition]  
We say that $\SX$ satisfies \emph{LVE condition at $p\in X$} if for every 
$j=l,\dots , d$ and every pair of real analytic arc germs $\gamma, \gamma' :[0,\epsilon)\to X^j$ at $p$,
 i.e. $\gamma(0)= \gamma'(0)=p$, every SLVF on $X^{j-1}\cup \gamma( [0,\epsilon))$ can be extended to a vector field on $X^{j-1}\cup \gamma( [0,\epsilon))\cup \gamma'( [0, \epsilon))$ satisfying the following condition: \\
 \emph{there is a constant $L$ such that for every $s$ sufficiently small this extension is an SLVF vector field, with Lipschitz constant $L$, 
  on $X^{j-1}\cup \gamma(s)\cup \gamma'(s)$}. 
 \end{defi}

\begin{rmk}
The following, a priori stronger condition, implies the LVE: 
 for every SLVF on $X^{j-1}\cup \gamma( [0,\epsilon))$ there is $\varepsilon '>0$  such that this vector field admits an extension that is  SLVF on  $X^{j-1}\cup \gamma( [0,\epsilon'))\cup \gamma'( [0, \epsilon'))$.
 \end{rmk}

We say that \emph{$\SX$ induces a Lipschitz stratification at $p\in X$} if there is an open neighbourhood $U$ of $p$ such that $\SX$ restricted to $U$ induces a Lipschitz stratification of $X\cap U$.  

\begin{prop}[LVE Criterion] \label{prop:LVE}
$\SX$ induces a Lipschitz stratification at $p\in X$ if and only if it satisfies the LVE condition at $p$. 
\end{prop}

\begin{proof}
We first recall the notions of a chain and Mostowski's Conditions.  We follow the approach of \cite{Lipschitz-Fourier} simplifying a little bit  the notation and exposition.  For slightly different but equivalent conditions see \cite{mostowski85,Lipschitz-review}.  One can simplify 
the proof below by using directly the valuative criteria of \cite{halupczok-yin2108} but we prefer to give a self-contained proof based 
on elementary computations given in the proofs of Proposition 1.2 and 1.5 of \cite{Lipschitz-Fourier}.

Fix $c>1$. A \emph {chain (more exactly, a $c$-chain)} 
for a point $q\in \mathring{X}^{j}$ is a strictly decreasing sequence of indices
$j= j_1, j_2, \ldots , j_r=l$ and a sequence of points $q_{m} \in \mathring{X}^{j_m}$ 
such that $q_{1}=q$ and $j_m$  is the greatest integer for which 
\begin{align*}
& \dist (q,X^{{k}}) \ge 2c^2   \dist (q,X^{j_{m}} ) \text{ for all }   k<j_m \\
& |q-q_{m}| \le c \dist (q,X^{{j_m}}).
\end{align*}
The condition $c>1$ is imposed only to ensure that every point $q\in X$ admits a chain.   A chain satisfies the following properties: 
\begin{enumerate}
	\item [(1)]
	$ \dist (q,X^{j_{m+1}}) \le 2^n c^{2n}  \dist (q,X^{j_{m}-1}) $,
\item [(2)]
$|q_{m}-q_{{m+1}}| \le 2^{n+1} c^{2(n+1)}  \dist (q,X^{j_{m}-1}) $,
\item [(3)]
$2\dist (q_m,X^{j_{m}-1})\ge \dist (q,X^{j_{m}-1})$.  
\end{enumerate}

Let $P_q:\R^n\to T_q\mathring X^j$ denote the orthogonal projection onto the tangent space  and $P_q^{\perp}=I-P_q$ the orthogonal projection onto the normal space $T_q^\perp\mathring X^j.$
We say that $\SX$ satisfies Mostowski's Conditions if there is a constant $C>0 $ such that for all chains $\{q_m\}_{m=1, \ldots , r}$ and all $2\le k\le r$:  
\begin{align}\label{eqn:M1}\tag{M1}
| P_{q_1}^{\perp} P_{q_{2}} \cdots P_{q_{k}} | \le C|q-q_2| / \dist (q,X^{j_{k}-1}).
\end{align}
If, further, $q'\in \mathring{X}^{j}$ and $|q-q'| \le (\frac 1 {2c} ) \dist (q, X^{j-1})$ then 
\begin{align}\label{eqn:M2}\tag{M2}
| (P_q - P_{q'}) P_{q_{2}} \cdots P_{q_{k}} |  \le C|q-q'| / \dist (q,X^{j_{k}-1}), 
\end{align}
in particular, 
\begin{align}\label{eqn:M3}\tag{M3}
| P_q - P_{q'}| \le C|q-q'| / \dist (q,X^{j_{1}-1}), 
\end{align}
where $\dist ( \cdot, \emptyset ) \equiv 1$.   

By Proposition 1.5 of \cite{Lipschitz-Fourier},  $\SX$ induces a Lipschitz stratification if and only if any of two equivalent conditions (i) and (ii) of Proposition \ref{prop:LipStr-charcterization} holds.   In particular the definition of Mostowski's stratification is independent of the choice of the constant $c>1$ used to define the chains.  

Clearly by Proposition \ref{prop:LipStr-charcterization} a Lipschitz stratification satisfies LVE condition at any point of $X$.  

Suppose that $\SX$ satisfies LVE condition at  $p$.  We show by induction on $j$ that $\SX$ induces a Lipschitz statification of $X^j$ at $p$, the case $j=l$ being obvious because $X^l$ is nonsingular.  Thus we suppose it for $X^{j-1}$ and prove for $X^j$.  Suppose the latter does not hold.  Then by a fairly straightforward application of the curve selection lemma there are real analytic arcs $q_m (s) : [0, \varepsilon) \to X^{j_m}$, $m=1, \ldots, r$, $j_1=j$, at $p$, that are c-chains of $q(s) = q_1(s)$ for $s\ne 0$, and possibly another arc $q' (s) : [0, \varepsilon) \to X^j$ satisfying $|q(s)-q'(s)| \le (\frac 1 {2c} ) \dist (q(s), X^{j-1})$ for $s\ne 0$, for which one of the conditions \eqref{eqn:M1},\eqref{eqn:M2} fails, that is it holds with the constant $C(s)\to \infty$ as $s\to 0$. 
Indeed, it follows from Lemma 6.2 of \cite{mostowski85}, that is stated in the complex analytic set-up, or from the valuative criteria of \cite{halupczok-yin2108}, where the authors even managed to get rid of the constant $c$ defining the chains.  

We will show that the existence of such arcs contradicts LVE condition. We may assume that the index $k$, given by the length of the expression on the left-hand side of \eqref{eqn:M1},\eqref{eqn:M2}, for which one of these conditions fails is minimal.  Suppose that this is the condition \eqref{eqn:M1}.  
Let us then put $\gamma' (s):= q(s)$ and $\gamma (s) := q_2(s)$. 
 Then adapting the proofs of Propositions 1.2 and 1.5 of \cite{Lipschitz-Fourier} and using LVE condition we show that there is a constant $C>0$, independent of $s$, such that \eqref{eqn:M1} holds along the family of arcs $q_m , m=1, \ldots, k$, that gives a contradiction. 

   Let $V _0= lim_{s\to 0} T_{q_k(s)} \mathring{X}^{j_k}$.  Then $\dim V_0 = j_k$.   
 Let $v\in V_0$, $|v|=1$. Then $x\to \dist (x, X^{j_k-1}) v$ is a Lipschitz vector field (on a neighborhood of $p$) with the Lipschitz constant $1$. By the proof of Proposition 1.2 of \cite{Lipschitz-Fourier} (extension of Lipschitz vector fields on a Lipschitz stratification), $x\to P_x (\dist (x, X^{j_k-1}) v)$ defines 
a Lipschitz vector field on ${X}^{j_k}$.  By inductive assumption on $j$, we extend it to an SLVF, denoted by  $w$, on $X^{j-1}$ and then by LVE condition   
to the image of $\gamma'$.  This gives, together with \eqref{eqn:M1} for $m<k$ and the standard inqualities (1-3) satisfied by the chains, 
\begin{align*}
| & P_{q_1(s)}^{\perp} P_{q_{2}(s)} \cdots P_{q_{k}(s)}  w(q_k(s))| =|  P_{q_1(s)}^{\perp} P_{q_{2}(s)} \cdots P_{q_{k-1}(s)}  w(q_k(s))| \\
 & 
\le 
| P_{q_1(s)}^{\perp} P_{q_{2}(s)} \cdots P_{q_{k-1}(s)} w(q_{k-1}(s))| 
+ | P_{q_1(s)}^{\perp} P_{q_{2}(s)} \cdots P_{q_{k-1}(s)} (w(q_{k}(s) -w(q_{k-1}(s))| \\
& \cdots \\
& \le 
\sum _{1\le s<k}| P_{q_1(s)}^{\perp} P_{q_{2}(s)} \cdots P_{q_{s}(s)} (w(q_s(s))- (w(q_{s+1}(s))) | \\ 
& \le C \sum _{1\le s<k} \frac {|q(s)-q_{2}(s)| } {\dist (q, X^{j_s -1})} |q_s(s)- q_{s+1}(s)| \le C' |q(s) - q_2(s)| .
\end{align*}
Note that if $k=2$ the first term of the RHS of the first inequality does not appear, otherwise everything is the same.

Since $w(q_k(s)) = \dist (q_k(s), X^{j_k-1}) P_{q_k(s) } v$ we get, by property (3) of the chains,  
\begin{align*}
| P_{q_1(s)}^{\perp} P_{q_{2}(s)} \cdots P_{q_{k}(s)} v |   \le C' |q(s) - q_2(s)| / \dist (q(s), X^{j_k-1}).
\end{align*}
Applying the above to a finite set of $v$ from an orthonormal basis of $V_0$, and taking into account that $|P_{q_{k}(s)}v -v| \le C |q_{k}(s)| \to 0 $, as $s\to 0$, we show that 
\eqref{eqn:M1} holds along this family of arcs contrary to our assumptions.
A similar argument, based on the second part of the proof of Proposition 1.5 of \cite{Lipschitz-Fourier} applies to the condition \eqref{eqn:M2}.
This ends the proof.
\end{proof}

\begin{rmk}
Proposition \ref{prop:LVE} holds in a more general o-minimal set-up when one assumes every $X^j$ to be definable, every 
$\mathring{X}^{j}$ to be a $C^2$ submanifold, and the arcs to be continuous and definable.  One can also restrict the LVE condition, 
Definition \ref{def:LVE}, to definable vector fields, because the extension of Lipschitz vector fields construction of Proposition 1.2 of \cite{Lipschitz-Fourier} preserves the definability, see Remark 1.4 of \cite{Lipschitz-review}.  
\end{rmk}

\bibliographystyle{siam}
\DIFdelbegin 
\DIFdelend \DIFaddbegin \bibliography{ZELip}
\DIFaddend 

\end{document}